\documentclass[11]{article}
\usepackage{graphicx} 
\usepackage[utf8]{inputenc}
\usepackage{amsmath,amsfonts,amsthm, amssymb, mathtools}
\usepackage{bm}
\usepackage{geometry}
\usepackage{palatino} 
\usepackage{bbm}
\usepackage{todonotes}
\usepackage{mathtools}
\usepackage[hidelinks]{hyperref}
\usepackage[normalem]{ulem}
\usepackage{titlesec}
\usepackage{scrextend}
\usepackage{bbm}

\deffootnote[1em]{1em}{1em}{\textsuperscript{\thefootnotemark}\,}
\allowdisplaybreaks

\newcommand\redout{\bgroup\markoverwith
{\textcolor{red}{\rule[.5ex]{2pt}{3pt}}}\ULon}
\usepackage{xcolor}

\renewcommand{\theparagraph}{\thesubsubsection.\arabic{paragraph}}

\titleformat{\paragraph}[runin]
{\normalfont\normalsize\bfseries}{\theparagraph}{1em}{}
\titlespacing*{\paragraph}{0pt}{1em}{1em}

\DeclareMathOperator{\PROB}{\mathbb{P}}
\DeclareMathOperator{\PR}{\mathbb{P}}
\DeclareMathOperator{\EXP}{\mathbb{E}}
\DeclareMathOperator{\Ex}{\mathbb{E}}
\DeclareMathOperator{\diag}{diag}
\DeclareMathOperator{\trace}{Trace}
\DeclareMathOperator{\res}{Res}
\DeclareMathOperator{\sr}{sr}

\def\QED{$\blacksquare$}
\def\endProof{\hfill \QED \vspace{1ex}}

\let\eps=\varepsilon
\let\theta=\vartheta
\let\rho=\varrho
\let\phi=\varphi

\newcommand{\blackdiamond}{%
  \mathord{%
    \sbox0{$\diamond$}%
    \resizebox{!}{2.5\ht0}{%
      \raisebox{\depth}{\rotatebox[origin=c]{45}{$\blacksquare$}}%
    }%
  }%
}

\newcommand{\blw}{w}
\newcommand{\fbl}{f^{(\blw)}}
\newcommand{\sm}{\setminus}
\newcommand{\trans}{\mathsf{T}}
\newcommand{\Rel}{\mathrm{Rel}}

\newcommand{\R}{\mathbb{R}}
\newcommand{\cL}{\mathcal{L}}

\newcommand{\M}{\mathcal{M}}

\newcommand{\bX}{\bm{X}}
\newcommand{\bZ}{\bm{Z}}
\newcommand{\bXi}{\bm{\Xi}}
\newcommand{\bmu}{\bm{\mu}}
\newcommand{\bell}{\boldsymbol{\ell}}
\newcommand{\bPsi}{\bm{\Psi}}
\newcommand{\bZeta}{\bm{Z}}

\newcommand{\ba}{\boldsymbol{a}}
\newcommand{\bu}{\boldsymbol{u}}
\newcommand{\bi}{\boldsymbol{i}}
\newcommand{\bmm}{\boldsymbol{m}}
\newcommand{\bb}{\boldsymbol{b}}
\newcommand{\bt}{\boldsymbol{t}}
\newcommand{\bj}{\boldsymbol{j}}
\newcommand{\bk}{\boldsymbol{k}}

\newcommand{\bxi}{\bm \xi}
\newcommand{\bx}{\bm x}

\newcommand{\bz}{\bm z}
\newcommand{\bv}{\bm v}
\newcommand{\bzeta}{\bm \zeta}
\renewcommand{\bmu}{{\bm \mu}}
\newcommand{\bpsi}{\bm \psi}
\newcommand\norm[1]{\lVert#1\rVert}
\newcommand{\levy}{\mathcal L}

\newcommand{\chr}[1]{\mathbbm{1}_{\{#1\}}}

\newcommand{\frob}{\mathsf{F}}

\newcommand{\event}{\mathcal{E}}
\newcommand{\qf}{Q}

\newcommand{\PROBp}[1]{\PROB\left\{#1\right\}}
\newcommand{\EXPp}[1]{\EXP\left\{#1\right\}}
\newcommand{\be}{\bm e}

\newcommand{\bdelta}{\bm \delta}
\newcommand{\bvs}[2]{{#1}_{#2}}

\newcommand{\maxoverball}[4]{\max_{\left(\bvs{#1}{#3}\in N_{\ell_{#3}}(#2_{#3})\right)_{#3 \in [#4]}}}

\newcommand{\YM}{\Phi} 
\newcommand{\ham}{\textup{H}}
\newcommand{\indset}[1]{[n]^{#1}}
\newcommand{\quot}[1]{/
_{#1}}
\newcommand{\ress}{r}
\newcommand{\indsetp}[1]{\mathcal{T}_{#1}}
\newcommand{\Gameikf}{\bm{\Game}_{i,k}f}
\newcommand{\GameNorm}{\lVert \Gameikf \rVert}

\newtheorem{theorem}{Theorem}[section]

\newtheorem{conjecture}[theorem]{Conjecture}

\newtheorem{lemma}[theorem]{Lemma}

\newtheorem{claim}[theorem]{Claim}
\newtheorem{example}[theorem]{Example}
\newtheorem{corollary}[theorem]{Corollary}
\newtheorem{remark}[theorem]{Remark}

\numberwithin{equation}{section}

\newcommand{\bij}{\varphi}
\newcommand{\Ext}{\mathrm{Ext}}

\title{Resilience of Rademacher chaos of low degree}
\author{Elad Aigner-Horev \and Daniel Rosenberg \and Roi Weiss}
\date{}

\begin{document}
\maketitle

\begin{abstract}
The {\em resilience} of a Rademacher chaos is the maximum number of adversarial sign-flips that the chaos can sustain without having its largest atom probability significantly altered. Inspired by probabilistic lower-bound guarantees for the resilience of linear Rademacher chaos (aka. resilience of the Littlewood-Offord problem), obtained by Bandeira, Ferber, and Kwan (Advances in Mathematics, Vol. $319$, $2017$), we provide probabilistic lower-bound guarantees for the resilience of Rademacher chaos of arbitrary degree; these being most meaningful provided that the degree is constant.  
\end{abstract}

\section{Introduction}
Given that a property $\mathcal{P}$ is upheld by a structure of $S$, say, a (hyper)graph or a matrix, a quintessential question in Combinatorics is how {\sl strongly} does $S$ possess $\mathcal{P}$? 
Throughout the years, the combinatorial community has cultivated a rich culture and numerous methods by which the above question is made precise and thus delivering accurate methodology for estimating the so called {\sl robustness} of properties. For a precise list of lines of research thus pursued, consult the excellent survey~\cite{Su16} put forth by Sudakov.



One highly influential notion used to gauge the robustness of properties is that of {\sl resilience}. Roughly put, in the combinatorial setting, this term is typically associated with the ability of a structure to sustain {\sl adversarial} noise whilst maintaining a property of interest. Notions of resilience date back to the works of~\cite{KV04,KSV02}; however, it was the seminal result of Sudakov and Vu~\cite{SV08} which ushered the systematic study of this notion with an emphasis on (hyper)graph properties. For instance, the resilience of the Hamiltonicity property received much attention; see, e.g.~\cite{LS12,M19} and references therein. Such investigations that have propagated into High-Dimensional Probability~\cite{BFK17,FLM21} constitute our inspiration underlying this work.  

\medskip
\noindent 
\textbf{I. Vu's global rank resilience conjecture.} The determination of the statistical ubiquity of the singularity phenomenon in square random matrices is a problem that has its origins in Physics and has captured the attention of several mathematical communities such as High-Dimensional Probability, High-Dimensional Statistics, Combinatorics, and Theoretical Machine Learning. For an accurate account of the development of this problem through the years consult e.g.~\cite{T12}. 

In this venue, questions pertaining to the study of the singularity phenomenon of random matrices whose entries are discretely distributed have persisted the longest; see e.g.~\cite{BVWM10,KKS95,Komlos67,RV08,TV06,TV07,Tik,Vu21} and references therein. A high point in this line of research is the seminal result of Tikhomirov~\cite{Tik}, resolving an outstanding (and natural) conjecture, asserting that 
\begin{equation}\label{eq:conjecture}
\Pr\left\{M \sim \M_n \; \text{is singular} \right\} = \left(1/2 + o_n(1)\right)^n,
\end{equation}
holds, where $\M_{n,m}$ denotes the distribution of $n \times m$ Rademacher\footnote{A random variable $X$ is said to have the {\em Rademacher} distribution provided $\PROBp{X = -1} = 1/2 = \PROBp{X=1}$.} matrices\footnote{A random matrix is referred to as a {\em Rademacher matrix} provided its entries are i.i.d. copies of a Rademacher random variable.}; if $m = n$, then $\M_n$  is written instead of $\M_{n,n}$. More broadly still, the aforementioned result asserts that Rademacher matrices $M \sim \M_{n,m}$ are highly likely to have full rank. Additional highlights seen in the pursuit of~\eqref{eq:conjecture}, that we choose to accentuate, are the seminal results of Rudelson and Vershynin~\cite{RV08} bringing the theory of {\em anti-concentration} and specifically the so called problem of Littlewood and Offord~\cite{LO38} (as well as Erd\H{o}s~\cite{Erdos45}) into the fold; details pertaining to the latter follow below.   

In view of the above results pertaining to the pursuit of~\eqref{eq:conjecture} and in the wake of the seminal results addressing the {\sl resilience} of random graphs~\cite{KV04,SV08}, Vu~\cite{Vu08,Vu21} asked {\sl how resilient is the full rank property of Rademacher matrices with respect to adversarial entry-flips\footnote{Flips sfrom $+1$ to $-1$ and vice versa.}?}
More precisely, for a matrix $M \in \{\pm 1\}^{n \times m}$, $m \geq n$, write $\res(M)$ to denote the least number of entry-flips such that if performed on $M$ would produce a matrix whose rank is strictly less than $n$. As any two vectors $\ba,\bb \in \{\pm 1\}^n$ can be made to satisfy $\ba = \bb$ or $\ba = -\bb$ using at most $n/2$ flips; the bound $\res(M) \leq n/2$ then holds for any  $M \in \{\pm 1\}^{n \times m}$.  

\begin{conjecture}({\em Vu's global rank resilience  conjecture~\cite[Conjecture~11.5]{Vu21})\label{conj:Vu}\\
$\res(M) = (1/2 +o(1))n$ holds a.a.s.\footnote{Asymptotically almost surely; meaning with probability tending to one.} whenever $M \sim \M_n$. }
\end{conjecture}

\noindent
Vu's conjecture ventures further than~\eqref{eq:conjecture}. Any exponential bound $p_n \leq \eps^n$, for some fixed $\eps \in (0,1)$, coupled with a simple union-bound argument, yields $\res(M) = \Omega_\eps(n/\log n)$ holding a.a.s. whenever $M \sim \M_n$.  The leading intuition supporting Vu's conjecture is that in order to bring about rank-deficiency, one may restrict all flips to a small number of rows. 

Vu's conjecture is known to hold true if the number of columns exceeds the number of rows by an additive factor of order $o(n)$. This by a result of Ferber, Luh, and McKinley~\cite{FLM21} asserting that $\res(M) \geq (1-\eps)m/2$ holds a.a.s. whenever $M \sim \M_{n,m}$, $m \geq n + n^{1-\eps/6}$, and $\eps >0$ is independent of $m$ and $n$.

\medskip
\noindent
\textbf{II. Resilience of the Littlewood-Offord problem.} Products of the form $\ba^\trans \bxi$ with $\ba \in \mathbb{R}^n$ fixed and $\bxi \in \{\pm 1\}^n$ being a Rademacher vector\footnote{A vector whose entries are i.i.d. copies of a Rademacher random variable.}, are ubiquitous throughout probability theory. A classical result by Littlewood and Offord~\cite{LO38}, strengthened by Erd\H{o}s~\cite{Erdos45}, asserts that the largest {\sl atom} probability $\rho(\ba) :=\rho_{\ba^\trans \bxi}(\ba):= \sup_{x \in \mathbb{R}} \PROB[\ba^\trans \bxi = x]$ satisfies $\rho(\ba) = O(\|\ba\|_0^{-1/2})$ whenever $\ba \in \mathbb{R}^n$, where $\|\ba\|_0 := |\{i \in [n]: \ba_i \neq 0\}|$ is the size of the support of $\ba$; some classical generalisations of this result can be seen in~\cite[Remark~4.2]{CTV06}. In particular, if $\ba \in (\mathbb{R} \sm \{0\})^n$, then $\rho(\ba) = O(n^{-1/2})$; a bound which is asymptotically best possible for the all ones vector. As stated repeatedly throughout the relevant literature, this result is quite surprising; indeed, if the entries of $\ba$ have the same order of magnitude, then one may use the Berry-Esseen CLT (see  e.g.~\cite{Vershynin}) in order to attain the same bound on $\rho(\ba)$. The Erd\H{o}s-Littlewood-Offord result imposes essentially nothing on $\ba$, leading to the distribution of $\ba^\trans \bxi$ being possibly quite ``far" from Gaussian. In that, their result asserts that the CLT-bound coincides with the worst case estimate, leading to the wondering whether additional (arithmetic) assumptions imposed on the coefficients vector $\ba$ should manifest themselves in lower atom probabilities. This wondering has been solidified in a fairly long and substantial chain of results (see e.g.~\cite{Erdos65,Hal77,RV08,SS65,TV09,TaoVuBook}) making the relationship between the arithmetic structure of $\ba$ and $\rho(\ba)$ quite precise. Highly influential in this venue is~\cite[Theorem~1.5]{RV08} put forth by Rudelson and Vershynin in that it characterises $\rho(\ba)$ through the so called LCD-parameter of a sequence (see~\cite{RV08} for details). 
 
Inspired by the aforementioned conjecture of Vu as well as the relation between the Littlewood-Offord problem and the singularity of Rademacher matrices unveiled in~\cite{RV08}, a work by Bandeira, Ferber, and Kwan~\cite{BFK17} studies the resilience of the products $\ba^\trans \bxi$; asking: {\em how many adversarial flips can the entries of $\bxi$ sustain without forcing concentration on a specific value?} For $\ba \in (\mathbb{R} \sm \{0\})^n$, $\bxi \in \{\pm 1\}^n$, and $x \in \mathbb{R}$, write $\res_x^{\ba}(\bxi):= d_H(\bxi,L^{\ba}_x)$ to denote the Hamming distance between $\bxi$ and the level set $\mathcal{L}^{\ba}_x := \{\bpsi \in \{\pm 1\}^n: \ba^\trans \bpsi = x\}$; if reaching $x$ is impossible, then $\res_x^{\ba}(\bxi) = \infty$. If $\res_x^{\ba}(\bxi) > \ress$, then $\bxi$ is said to be $\ress$-{\em resilient} with respect to (the event) $\{\bpsi \in \{\pm 1\}^n: \ba^\trans \bpsi \neq x\}$. Roughly stated, $\res_x^{\ba}(\bxi) = \Omega(\log \log n)$ holds a.a.s. for any $x \in \mathbb{R}$ and $\ba \in (\mathbb{R}\sm 0)^n$~\cite[Theorem~1.8]{BFK17} with tightness established in~\cite[Theorem~1.7]{BFK17}.

\medskip

These results by Bandeira, Ferber, and Kwan constitute our main source of inspiration.

\subsection{Main results}

In this section, we state our main result, namely Theorem~\ref{thm:main-poly}; its formulation is provided in Section~\ref{sec:statements} and a discussion as to its impact is delegated to Section~\ref{sec:analysis}. 

\bigskip

A function $f: \mathbb{R}^{n_1} \times \cdots \times \mathbb{R}^{n_d} \to \mathbb{R}$ of the form
\begin{equation}\label{eq:poly-def}
f(\bx_1,\ldots,\bx_d) := \sum_{\bi \in \indsetp{d}} f_{\bi} \cdot (\bx_1)_{\bi_1}\cdots (\bx_d)_{\bi_d},
\end{equation}
where $\indsetp{d} := \Pi_{i\in[d]} [n_i]$ and $f_{\bi} \in \mathbb{R}$, is referred to as a (real) {\em multilinear polynomial/chaos of degree/order $d$}. The coefficients of $f$, namely $(f_{\bi})_{\bi \in \indsetp{d}}$, form a $d$-mode tensor and we use the latter notation to abbreviate~\eqref{eq:poly-def}. 
In that, it is beneficial to introduce the additional notation, namely $f(\bx_1 \circ \cdots \circ \bx_d)$,  along side the functional one, i.e. $f(\bx_1,\ldots,\bx_d)$, where $\circ$ denotes the outer product operation. 
We write $f\not\equiv 0$ to denote that at least one of the coefficients of $f$ is non-zero. Given independent Rademacher vectors\footnote{A vector $\bxi \in \{\pm 1\}^n$ whose entries are independent Rademacher random variables is called a {\em Rademacher vector}.}
$\bxi_1\in\{\pm1\}^{n_1},\dots,\bxi_d\in\{\pm1\}^{n_d}$,
write $\bXi := \bxi_1 \circ \cdots \circ \bxi_d$ and call $f(\bXi)$ a {\em decoupled Rademacher polynomial/chaos of degree/order $d$}. 

\bigskip

In view of the aforementioned work of Bandeira, Ferber, and Kwan~\cite{BFK17} for linear Rademacher chaos, a natural follow-up question to pose is: 
\begin{center}
{\sl how resilient are (decoupled) Rademacher chaos of arbitrary order?} 
\end{center}
Given a polynomial $(f_{\bi})_{\bi \in \indsetp{d}}$ as in~\eqref{eq:poly-def}, a real $x \in \mathbb{R}$, as well as  vectors 
$\bv_1\in\{\pm1\}^{n_1},\dots,\bv_d\in\{\pm1\}^{n_d}$,
write 
$$
\res_x^f(\bv_1,\ldots,\bv_d) := d_H\left((\bv_1,\ldots,\bv_d), \mathcal{L}_x^f\right) := \min_{(\bu_1,\ldots,\bu_d) \in \mathcal{L}_x^f} d_H\big((\bv_1,\ldots,\bv_d),(\bu_1,\ldots,\bu_d)\big),
$$
where 
$$
\mathcal{L}_x^f := \Big\{(\bu_1,\ldots,\bu_d) \in
 \{\pm 1\}^{n_1} \times \cdots \times \{\pm 1\}^{n_d}
: f(\bu_1,\ldots,\bu_d) = x\Big\}
$$
and $d_H\big((\bv_1,\ldots,\bv_d),(\bu_1,\ldots,\bu_d)\big)$ denotes the Hamming distance between the ensembles $(\bv_1,\ldots,\bv_d)$ and $(\bu_1,\ldots,\bu_d)$ taken as vectors in 
$\{\pm 1\}^{\sum_{i=1}^d n_i}$.
In that, $\res_x^f(\bv_1,\ldots,\bv_d)$ is the least amount of sign-flips that if performed on the members of $\bv_1,\ldots,\bv_d$ would generate an ensemble of vectors over which $f$ assumes the value $x \in \mathbb{R}$. If $\res_x^f(\bv_1,\ldots,\bv_d) > \ress$, for some $\ress \in \mathbb{N}$, then $\bv_1,\ldots,\bv_d$ is said to be $\ress$-{\em resilient with respect to $x \in \mathbb{R}$ along $f$}. If $\inf_{x \in \mathbb{R}} \res_x^f(\bv_1,\ldots,\bv_d) > \ress$, then $(\bv_1,\ldots,\bv_d)$ is said to be $\ress$-{\em resilient} (along $f$). If all vectors $(\bv_i)_{i\in [d]}$ coincide into a single vector, namely $\bz$, then we write $\res_x^f(\bz)$ instead of $\res_x^f(\bz,\ldots,\bz)$; necessary adaptations of the level sets and Hamming distance apply.

For a Rademacher tensor $\bXi$, the quantity $\res_x^f(\bXi)$ is its {\em resilience} with respect to the event $\{f(\bXi) \neq x\}$. Given $\ress \in \mathbb{N}$, upper bounds on 
\begin{equation}\label{eq:general-form}
\sup_{x\in \mathbb{R}} \PROBp{\res_x^f(\bXi) \leq \ress}
\end{equation}
bound the probability that $\bXi$ is not $\ress$-resilient.
Motivated by comfort, we sometime abuse the notation set here and attribute resilience to the polynomial itself when we write that $f$ or $f(\bXi)$ is $\ress$-resilient (or not). All of our main results, stated in the next section, deliver upper bound on~\eqref{eq:general-form} for various choices of $f$. 

To derive probabilistic lower-bound guarantees on the resilience of $f(\bXi)$ using our results, we seek the ``largest'' $r \in [n_{\max}]$, where $n_{\max}:= \max_i n_i$, for which 
$$
\sup_{x\in \mathbb{R}} \PROBp{\res_x^f(\bXi) \leq \ress} = o(1)
$$
can be shown to hold using our results. For any such $r$ identified, we then claim that the resilience of $f(\bXi)$ is a.a.s.\ at least $r$. More precisely, given $g:\mathbb N\to\mathbb N$, we say that $\res_x^f(\bXi) = \Omega(g(n))$ holds a.a.s.\ for all $x\in\R$ if for any $\ress(n)=o(g(n))$, the bound
$
\sup_{x\in \mathbb{R}} \PROBp{\res_x^f(\bXi) \leq \ress(n)} = o(1)
$
holds. 

\subsubsection{Statement of the main result}\label{sec:statements}

\paragraph{Outline.} As declared above, our main result is Theorem~\ref{thm:main-poly}. It pertains to the resilience of arbitrary degree (decoupled) Rademacher chaos. 
As the formulation of Theorem~\ref{thm:main-poly} requires a certain amount of effort, we propose a {\sl\bf model result}, namely Theorem~\ref{thm::main-bf}, 
targeting the resilience of (decoupled) Rademacher chaos of order two. This model result admits a significantly simpler formulation and thus allows us to deliver a concrete result to the reader avoiding a hefty preparation at an early stage. Moreover, certain key ideas employed in the proof of our main result, namely Theorem~\ref{thm:main-poly}, admit a simpler presentation in the proof of Theorem~\ref{thm::main-bf}. With this in mind, the latter is presented first albeit not completely identical to Theorem~\ref{thm:main-poly}. The proofs of the two results can be read independently and the reader interested in the more general result only can skip Theorem~\ref{thm::main-bf} altogether. 
 
For the interested reader, allow us to remark here that proof-wise the crucial differences between the model result, namely Theorem~\ref{thm::main-bf}, and our main result, namely  Theorem~\ref{thm:main-poly}, lie in the methods employed carrying out the same conceptual framework driving the proofs of each of these results. To prove the former, we employ Dudley's maximal inequality for sub-gaussian processes~\cite[Lemma~5.1]{van2014probability} (see Theorem~\ref{thm::dudley-max}), the Kolmogorov-Rogozin anti-concentration inequality~\cite{kolmogorov1958proprietes,rogozin1961estimate,rogozin1961increase} (see Theorem~\ref{thm:KRI}), as well as the Hanson-Wright concentration inequality~\cite[Theorem~2.1]{rudelson2013hanson}. The transition from order two Rademacher chaos to ones with arbitrary degree leads, somewhat naturally, to losing the ability to appeal to tools utilising sub-gaussianity; in that, the classical results of Dudley and Hanson-Wright are rendered inadequate in the new setting. For instance, the latter is replaced by an appeal to a concentration result generalising the Hanson-Wright inequality and put forth by Adamczak and Wolff~\cite[Theorem~1.4]{AW15} (see Theorem~\ref{thm:Adamcask}).

Similar to many of the anti-concentration results mentioned in the introduction, our resilience results are instance-dependent as well and consequently do not provide worst case estimates for the resilience. If the degree of the Rademacher chaos is fixed, then our results provide a way to efficiently compute probabilistic lower-bound guarantees for the resilience of said chaos. 

Further appreciation for our main result, can be attained through a discussion of the aforementioned theorems found in Section~\ref{sec:analysis} pursuing resilience guarantees  along an arc starting from the identity matrix, passing through block-diagonal matrices, and then culminating in block-diagonal high-dimensional tensors. 


\paragraph{A model result: Decoupled bilinear Rademacher chaos} For $M \in \mathbb{R}^{n \times m}$, set $\|M\|_\infty := \max_{i,j} |M_{ij}|$ as well as 
$$
\|M\|_{\infty,2}  := \max_{\bmm \in \mathrm{Rows}(M)} \|\bmm\|_2 \quad \text{and} \quad
\|M\|_{\infty,0}  :=\max_{i\in[n]} \big|\{M_{ij}\neq 0: j\in [n]\}\big|, 
$$
where $\mathrm{Rows}(M)$ denotes the set of rows of $M$. By the triangle inequality,   
$$
\mathrm{diam}\big(\mathrm{Rows}(M)\big) := \max_{\ba,\bb \in \mathrm{Rows}(M)} \|\ba-\bb\|_2 \leq 2 \|M\|_{\infty,2}
$$ 
holds, where $\mathrm{diam}(X)$ denotes the diameter of $X \subseteq \mathbb{R}^n$. The quantity $\|M\|_{\infty,0}$ is the maximal support amongst the rows of $M$.
All of the above quantities are defined for the columns of $M$ by substituting $M$ with $M^\trans$. In that, write 
\begin{align*}
\mathrm{maxsupp}(M) &:= \max \big\{ \|M\|_{\infty,0}, \|M^\trans\|_{\infty,0} \big\}\\
\mathrm{maxdiam}(M) &:= \max \big\{ \|M\|_{\infty,2}, \|M^\trans\|_{\infty,2} \big\}.
\end{align*}
The {\em stable rank} of $M$ is given by $\sr(M) := \|M\|_\frob^2/\|M\|_2^2$, where $\|M\|_\frob$ denotes its Frobenius norm and $\|M\|_2$ denotes its spectral norm. 
The following quantities 
\begin{align*}
f(M,n) & := \frac{\min\Big\{ \mathrm{maxsupp}(M) \cdot \|M\|_\infty\;,\; \sqrt{\log n} \cdot \mathrm{maxdiam}(M)\Big\}}{\|M\|_\frob},\\
\qquad g(M,\ress) &:= \frac{\min\Big\{\ress\;,\;\mathrm{maxsupp}(M)\Big\} \cdot \|M\|_\infty}{\|M\|_\frob}
\end{align*}
arise in the formulation of our results below. Insight into these is offered in Section~\ref{sec:analysis}; at this stage let us make do with noting that 
both these quantities can be efficiently computed given $M$. 

\medskip

With the above notation in place, we are in position to state our model result pertaining to the resilience of bilinear Rademacher chaos. 

\begin{theorem}
\label{thm::main-bf}
Let $0 \not\equiv M \in \R^{n \times m}$. There are constants $c_1,c_2,c_3>0$ such that for any integer $\ress\in [n]$,
\begin{equation}\label{eq:prop_bf}
\sup_{x\in\R}\PROBp{\res_x^M(\bpsi,\bxi) \leq \ress}
 \leq c_1 r\cdot f(M,n) + c_2 \ress \cdot g(M,\ress) + \exp{\left(-c_3\sr(M)\right)},
\end{equation}
where $\bpsi \in \{\pm 1\}^{m}$ and $\bxi \in \{\pm 1\}^{n}$ are independent Rademacher vectors and the chaos has the form $\bpsi^\trans M \bxi$.
\end{theorem}


To mitigate the (still) rather abstract formulation of Theorem~\ref{thm::main-bf}, we propose the following more transparent corollary of the latter, deduction of which we delegate to Appendix~\ref{par:deduce}. Corollary~\ref{cor:transparent} distinguishes between so called {\sl sparse} and {\sl dense} matrices and states resilience guarantees for each regime. This distinction is afforded to us by the minimisations seen in the terms $f(M,n)$ and $g(M,r)$ defined above. In the sparse regime, tightness of the estimates is reached (see Section~\ref{par:identity} for details). 

\begin{corollary}\label{cor:transparent}
Let $0 \not\equiv M \in \mathbb{R}^{m \times n}$ satisfying $\sr(M) = \omega(1)$  as well as 
$$
\|M\|_\infty = 1,\; \|M\|_{\infty,0} = \mathrm{maxsup}(M), \; \text{and}\; \|M\|_{\infty,2} = \mathrm{maxdiam}(M)
$$
be given. 
Then, 
\begin{enumerate}
    \item {\bf Sparse regime: $\boldsymbol{\|M\|_{\infty,0} \leq \log n}$.} A.a.s.\ $M$ has resilience $\Omega\left(\frac{\|M\|_\frob}{\|M\|_{\infty,0}}\right)$ and this is asymptotically tight\footnote{See Section~\ref{par:identity}.} for diagonal matrices whose diagonal entries have the same order of magnitude.   

    \item {\bf Dense regime: $\boldsymbol{\|M\|_{\infty,0} \geq \log n}$.} 

    \begin{enumerate}
        \item If $\|M\|_\frob > \|M\|_{\infty,0}^2$, then resilience is a.a.s. $\Omega\left(\frac{\|M\|_\frob}{\|M\|_{\infty,0}}\right)$;

        \item Otherwise resilience is a.a.s. 
        $$
        \Omega \left( \min \left\{\frac{\|M\|_\frob}{\sqrt{\|M\|_{\infty,0}\log n}} \;,\; \sqrt{\|M\|_\frob} \right\} \right).
        $$
    \end{enumerate}
    
\end{enumerate}
\end{corollary}

\begin{remark}\label{rem:limitation}{\bf (Limits of our methods)}
{\em  Subject to $\|M\|_\infty =1$ seen in Corollary~\ref{cor:transparent}, the trivial bound $\norm{M}_\frob \leq \sqrt{n\norm{M}_{\infty,0}}$ holds implying that the $\Omega(\sqrt{n})$ bound for resilience in 
Corollary~\ref{cor:transparent} cannot be breached by our methods and we suspect that its estimates are far from the truth as can be seen in Section~\ref{par:identity} if, say, the stable rank assumption is removed. More generally, and as mentioned above, our methods provide instance-dependent estimates. It would be interesting to improve upon our estimates throughout.
}
\end{remark}

\begin{remark} {\bf (Stable rank)}
{\em Theorem~\ref{thm::main-bf} is meaningful provided $\sr(M) = \omega(1)$, where $M$ is per those theorems. This condition is incurred through an appeal to the Hanson-Wright inequality. The wide-spread use of the latter throughout High-Dimensional Probability, Statistical Learning, and Compressed Sensing, see e.g.~\cite{KRu18,rudelson2013hanson,Vershynin} and references therein, renders the condition $\sr(M) = \omega(1)$ to be fairly standard in these venues. Indeed, such an imposition can be seen in the commonly used small-ball probability inequality~\cite[Corollary~2.4]{rudelson2013hanson} and variants thereof; in Smoothed Analysis results such as~\cite[Theorem~3.1]{KR19} (and references therein) and so on.}    
\end{remark}

\begin{remark}\label{rem:quadratic}{\bf (Resilience of quadratic Rademacher chaos)}
{\em In Appendix~\ref{sec:quadratic-app}, we utilise Theorem~\ref{thm::main-bf}, an adaptation of a  decoupling argument seen in~\cite{CTV06}, as well as Dudley's maximal {\sl tail} inequality~\cite[Lemma~5.2]{van2014probability} for sub-gaussian processes (see Theorem~\ref{thm::Dudley-tail}), in order to provide probabilistic estimates for quadratic Rademacher chaos. For details, consult Theorem~\ref{thm:qf}.}    
\end{remark}

\paragraph{Main result: Decoupled Rademacher chaos of arbitrary degree}  
Our main result, namely Theorem~\ref{thm::main-bf} is stated next; its statement requires preparation. Let $2 \leq d\in \mathbb{N}$ be given.  A degree~$d$ polynomial given by $(f_{\bi})_{\bi \in \indset{d}}$ as in~\eqref{eq:poly-def} satisfying $n_i = n$ for every $i \in [d]$ has its {\em dimensions} captured by the set $[d]$ where for each dimension $i \in [d]$ there are $n$ {\em directions} associated with the set $[n]$. For such a polynomial, write $\|f\|_\frob^2 := \sum_{\bi \in \indset{d}} f_{\bi}^2$ and $\|f\|_\infty := \max_{\bi \in \indset{d}} |f_{\bi}|$. Given independent Rademacher vectors $(\bxi_i \in [\pm 1]^{n})_{i\in [d]}$, recall that we write $\bXi = \bxi_1 \circ \cdots \circ \bxi_d$ to denote the rank one decoupled Rademacher tensor. Given $I \subseteq [d]$, write 
$\bXi_{\quot{I}}$ (pronounced ``Xi quotient $I$") to denote the (partial decoupled) Rademacher tensor obtained by omitting the members of $(\bxi_i: i \in I)$ from $\bXi$; write $\bXi_{\quot i}$ instead of $\bXi_{\quot{\{i\}}}$. 

\medskip
\noindent
{\tt Partial derivatives through matrixisation.} The first term required for the statement of Theorem~\ref{thm:main-poly} is $\GameNorm$ defined for every $i \in [d]$ and every {\sl even} $k \in [2(d-1)]$. As the symbol $\bm{\Game}$ suggests, this term is associated with the partial derivatives of order $k$ ($k$th-{\em derivatives}, hereafter) of a certain (Rademacher) polynomial related to $f$ (and disclosed momentarily). In fact, the quantity $\GameNorm$ serves as an upper bound on certain operator norms of a tensor housing the expectations of said derivatives. This related polynomial arises through a specific matrixisation of $f$ defined next. Partial derivatives of said polynomial are encountered through our appeal to the aforementioned Adamczak-Wolff concentration result (see Theorem~\ref{thm:Adamcask}).

\medskip

Let $(f_{\bi})_{\bi \in \indset{d}}$ be as in~\eqref{eq:poly-def}. 
Given $i \in [d]$, define $\mathcal{A}_i := \mathcal{A}^{(f)}_i$ to be the $n \times n^{d-1}$-matrix whose entries are the coefficients of $f$ arranged in $\mathcal{A}_i$ according to 
\begin{equation}\label{eq:order-coeff}
(\mathcal{A}_i)_{j,{\bm \ell}} := f_{j \overset{i}{\hookrightarrow} {\bm \ell}},
\end{equation}
whenever $j \in [n]$, ${\bm \ell} := (\ell_1,\ldots,\ell_{i-1},\ell_{i+1},\ldots,\ell_d) \in [n]^{d-1}$, and where $j \overset{i}{\hookrightarrow} {\bm \ell}$ denotes the $d$-tuple given by 
$$
j \overset{i}{\hookrightarrow} {\bm \ell} := (\ell_1,\ldots,\ell_{i-1},j,\ell_{i+1},\ldots,\ell_d).
$$
The matrix $\mathcal{A}_i$ allows us to isolate $\bxi_i$ from $\bXi$ and write $f(\bXi)\equiv\bxi_i^\trans \mathcal{A}_i \mathrm{vec}({\bXi_{\quot{i}}})$ for all $i\in[d]$, where $\mathrm{vec}(\cdot)$ denotes the classical vectorisation operation for tensorial products.

\medskip

Proceeding to the definition of $\GameNorm$ for $i \in [d]$ and an {\sl even} $k \in [2(d-1)]$, start by 
setting the mapping $\bij_i:[n]^{d-1} \to \big(([d]\sm\{i\})\times[n]\big)^{d-1}$, given by 
\begin{equation}\label{eq:phi}
\bij_i(\bj):= \big( (1,\bj_1), \dots, (i-1,\bj_{i-1}), (i+1,\bj_i),\dots, (d,\bj_{d-1
})\big)
\end{equation}
whenever $\bj\in[n]^{d-1}$; in that, $\bij_i(\bj)$ is a sequence of ordered pairs first member of each indicates the index of a vector $\bxi_\ell$ with $\ell \in [d] \sm \{i\}$; the second member of each pair records the entry of $\bxi_\ell$ to be referenced through $\bj$. 

The mapping $\bij_i$ is utilised as follows. Given an even integer $k \in 2(d-1)$, $i \in [d]$, as well as a tuple of pairs $\bell \in \big(([d] \sm \{i\}) \times [n]\big)^k$, set
\begin{equation}\label{eq:ext}
\Ext_{i,k}(\bell) := \left\{(\bj,\bk) \in [n]^{d-1} \times [n]^{d-1}: \big(\bij_i(\bj) \cup \bij_i(\bk) \big) \sm \big(\bij_i(\bj) \cap \bij_i(\bk) \big) = \bell \right\}
\end{equation}
to denote the {\em extension set of $\bell$} (with respect to $i$ and $k$). Indeed, through the lens of the mapping $\bij_i$, sequences $(\bj,\bk)$ found in the extension set of $\bell$  have their intersection (or common subsequence subject to $\bij_i$) extended by $\bell$. 
For an even $k \in  [2(d-1)]$, extensions are only considered for so called {\em relevant} tuples, by which we mean tuples of the form 
$$
\big((b_1,c_1),\dots,(b_{\frac k 2},c_{\frac k 2}),(b_1,c'_1),\dots,(b_{\frac k 2},c'_{\frac k 2})\big)\in \big([d]\setminus\{i\}\times[n]\big)^k
$$
equipped with the property that all $b$-elements are distinct of one another and for every $i \in [k]$ the corresponding entries, namely $c_i$ and $c'_i$, do not coincide with one another. We then set 
\begin{align} \label{eq:Rel}
    \Rel_{i,k} = \left\{\bell \in \big([d]\setminus\{i\}\times[n]\big)^k: \bell \; \text{is relevant} \right\}.
\end{align}

With the above notation in place, define 
\begin{equation}\label{eq:partial}
\GameNorm^2 := \sum_{\bell\in \Rel_{i,k}}
\left(
\sum_{(\bj,\bk) \in \Ext_{i,k}(\bell)} 
(\mathcal{A}_i^\trans\mathcal{A}_i)_{\bj \bk}
\right)^2. 
\end{equation}
Roughly put, the external sum defining $\GameNorm$ defines a relevant extension sequence; the inner sum then ranges over all pairs $(\bj,\bk) \in [n]^{d-1} \times [n]^{d-1}$ that beyond their intersection (as seen through $\bij_i$) coincide with $\bell$ and in that sense $\bell$ extends their intersection (defined by $\bij_i$).  

The combinatorial quantity $\GameNorm$ has an enigmatic feel to it and deserves much explaining. True understanding of this quantity can be obtained at its origin, namely the proof of Lemma~\ref{lem:v-concentration}, where the concentration properties of the (quadratic-looking) Rademacher polynomial 
$$\mathrm{vec}(\bXi_{\quot{i}})^\trans \mathcal{A}_i^\trans\mathcal{A}_i \mathrm{vec}(\bXi_{\quot{i}})$$ related to $f$ are considered. Prior to this lemma, one way to mitigate the enigmatic nature of $\GameNorm$ at this preliminary and declarative stage is to consider Appendix~\ref{par:block-diagonal} where we apply Theorem~\ref{thm:main-poly} to $d$-mode tensors; there a direct handling of the quantity $\GameNorm$ can be seen thus removing some of the obfuscation accompanying this quantity.

The matrix $\mathcal{A}_i^\trans\mathcal{A}_i$ appearing in the definition of $\GameNorm$ is referred to as the {\em correlation matrix of $f$ at dimension $i$}; in that, given $\bu,\bv \in [n]^{d-1}$, entries of the form $\left(\mathcal{A}_i^\trans \mathcal{A}_i\right)_{\bu \bv}$ capture the inner products between fibres of $f$ indicated by the indices $\bu$ and $\bv$. The need for the mapping $\phi_i$, seen in the definition of $\GameNorm$, arises from a certain technical nuisance divulged in the proof of Lemma~\ref{lem:v-concentration}.

\medskip
\noindent
{\tt Restrictions.} Given a set of dimensions $\emptyset \neq I \subseteq [d]$ as well as a tuple of directions $d_I := (d_i: i \in I) \in \indset{|I|}$, define the {\em selection} function $\bdelta_{I,d_I}(\cdot)$ which maps a dimension $j \in [d]$ to
\begin{equation}\label{eq:selection}
\bdelta_{I,d_I}(j) := 
\begin{cases}
    \be_{d_i}, & j = i \in I;\\
    \bxi_j, & j \notin I, 
\end{cases}
\end{equation}
where $\be_{d_i}$ denotes a standard vector in $\mathbb{R}^{n}$ having all entries but its $d_i$th entry set to zero and the $d_i$th entry set to one. The (decoupled) Rademacher chaos given by 
\begin{equation}\label{eq:restriction}
f_{I,d_I}(\bXi_{\quot{I}}) : = f(\bdelta_{I,d_I}(1),\ldots,\bdelta_{I,d_I}(d))
\end{equation}
is said to be a {\em restriction} of $f$ in the sense that in the dimensions specified by $I$ no randomness is retained and instead the directions specified by $d_I$ are fixed (through the coefficients). In that, $f_{I,d_I}(\bXi_{\quot{I}})$
does not depend on the vectors $(\bxi_i : i \in I)$ which are suppressed in it, so to speak.
The restriction $f_{I,d_I}$ can be viewed as a
$(d-|I|)$-mode tensor by considering its coefficients inherited from $f$ indices of which at dimensions $I$ are fixed to $d_I$. In this regard, the norms $\|f_{I,d_I}\|_\infty$ and $\|f_{I,d_I}\|_\frob$ are defined analogously to their respective counterparts, namely $\|f\|_{\infty}$ and $\|f\|_\frob$, over the coefficients of the restriction $f_{I,d_I}$. Examples of restrictions can be seen in Section~\ref{sec:tensor-restrictions} as well as Appendix~\ref{par:block-diagonal}.

\medskip
Our main result reads as follows. 

\begin{theorem}\label{thm:main-poly}
Let $2 \leq d \in \mathbb{N}$ be given and let $(f_{\bi})_{\bi \in \indsetp{d}}$ be as in~\eqref{eq:poly-def} such that $f \not\equiv 0$ and satisfying $n_i = n = \omega(d^d)$ for every $i \in [d]$. Let $\bxi_1,\ldots,\bxi_d$ be independent Rademacher vectors conformal with the dimensions of $f$. Then, 
{\small 
\begin{align}
\sup_{x \in \mathbb{R}} & \PROBp{\res_x^f(\bXi) \leq r} = \label{eq:main-poly-res}\\ 
& 
O_d(1) \log^{\frac{d}{2}}(n) 
\sum_{I \subseteq [d]: I\neq\emptyset}
r^{|I|}
     \max_{d_I \in \indset{|I|}}
     \frac{\norm{f_{I,d_I}}_\frob}{\|f\|_\frob}
  + 
  O_d(1)\exp \left(-\Omega_d(1) \left( \frac{ \|f\|^2_\frob}{\lVert\bm\Game f\rVert}\right)^{ 1/{\Theta_d(1)}} \right)\nonumber
\end{align}
}
holds, whenever $r \in [n]$, where 
$$
\lVert\bm\Game f\rVert :=  \max_{i \in [d]}\,\max_{\substack{k \in [2(d-1)]  \\ k\; \text{even}}} \GameNorm,
$$
and where $O_d(1)$, $\Omega_d(1)$, and $\Theta_d(1)$ denote quantities dependent solely on $d$. 
\end{theorem}


To gain insight into the formulation of Theorem~\ref{thm:main-poly}, one should draw a comparison between the formulation of the latter and that of Theorem~\ref{thm::main-bf} - the model result. Additional aid in unpacking Theorem~\ref{thm:main-poly} can be found below, in Section~\ref{para:block-diag-tensors}, where we provide a working example employing Theorem~\ref{thm:main-poly} over block-diagonal tensors. Through this working example, one gains access to the various components seen on the right hand of~\eqref{eq:main-poly-res} disclosing their respective roles in a more vivid manner. 

Whilst Theorem~\ref{thm:main-poly} fits a fairly general setting, we suspect that the estimates it provides for resilience are not optimal. 
It would be interesting to have our estimates reachable through Theorem~\ref{thm:main-poly} improved upon.

\begin{remark}
{\em The quantities 
$O_d(1)$ appearing in~\eqref{eq:main-poly-res} are exponential in $d$. The $\Theta_d(1)$ quantity is linear in $d$. A more accurate formulation of Theorem~\ref{thm:main-poly} is provided in Theorem~\ref{thm:main-poly-2} where the raw nature of these quantities can be seen.}
\end{remark}

\begin{remark}
{\em The sum
$$
\sum_{I \subseteq [d]: I\neq\emptyset}
r^{|I|}\cdot \max_{d_I \in \indset{|I|}}
\frac{\norm{f_{I,d_I}}_\frob}{\norm{f}_\frob}
$$ 
appearing on the right hand side of~\eqref{eq:main-poly-res} exhibits a certain trade-off behaviour. The larger is the set $I$ in the sum, the larger is the exponent of $\ress$ and the smaller is the quantity $\max_{d_I \in \indset{|I|}} \|f_{I,d_I}\|_\frob$.}
\end{remark}

\begin{remark}\label{rem:Vershynin}
{\em The term $\GameNorm$ introduced in~\eqref{eq:partial} is incurred through an appeal to the aforementioned concentration result of Adamczak and Wolff~\cite[Theorem~1.4]{AW15} (see Theorem~\ref{thm:Adamcask} for an abridged and significantly weaker formulation). The latter entails certain operator norms of expectations of partial derivatives be handled. The Adamczak-Wolff result, whilst best possible in its venue, is accompanied with a high level of abstraction, as the proof of Lemma~\ref{lem:v-concentration} illustrates, and is not easy to wield, so to speak. Alternatives to this generalisation of the Hanson-Wright inequality that offer some more ease of use are known. One such alternative is a result by Verhsynin~\cite[Theorem~1.4]{vershynin2020concentrationinequalitiesrandomtensors}; another is an improvement on the latter attained by  Bamberger, Krahmer, and Ward~\cite[Theorem~2.1]{BKW22}. 

Alas, these more readily deployable results are inadequate for our needs hence our appeal to the Adamczak-Wolff result. Roughly put, employing the former in our argument would render the exponential term seen on the right hand side of~\eqref{eq:main-poly-res} to potentially be as large as 
\begin{equation}\label{eq:exp-vershynin}
\exp \left(- C_d \frac{\|f\|^2_{\frob}}{n^{d-2} \|\mathcal{A}_i^\trans \mathcal{A}_i\|_2} \right),
\end{equation}
where $C_d >0$ depends on $d$, 
and it is this factor of $n^{d-2}$ seen here that our analysis cannot bear; indeed, since $\mathcal{A}_i$ has rank at most $n$ and $\|f\|^2_\frob = \|\mathcal{A}_i\|^2_\frob\leq n \|\mathcal{A}_i\|^2_2 =n\|\mathcal{A}_i^\trans\mathcal{A}_i\|_2$, the bound \eqref{eq:exp-vershynin} is vacuous for $d\geq 3$.   
The Adamczak-Wolff result allows us to avoid this term. An example of these savings, afforded to us by the Adamczak-Wolff result, can be seen in Claim~\ref{clm:block-diagonal}, where Theorem~\ref{thm:main-poly} is used in order to provide resilience guarantees for $d$-mode
block-diagonal Rademacher tensors.}  
\end{remark}

\bigskip
\noindent
{\bf Organisation.} Theorem~\ref{thm::main-bf} - our model result -  is proved in Section~\ref{sec:quad-bi}. Theorem~\ref{thm:main-poly} - our main result - is proved in Section~\ref{sec:main-poly}. 

\bigskip
\noindent
{\bf Acknowledgements.} We would like to thank Matthew Kwan and Lisa Sauermann for their comments on a previous version of the manuscript.

\subsubsection{Discussion}\label{sec:analysis}

In this section, we employ our results for obtaining lower bound resilience guarantees along an arc starting with the identity matrix, passing through block-diagonal matrices, and culminating in block-diagonal high-degree tensors.

\paragraph{The identity matrix} \label{par:identity}
 The decoupled bilinear Rademacher chaos  $\bpsi^\trans I_n \bxi = \sum_{i=1}^n \bpsi_i \bxi_i$ has the same distribution as 
$\sum_{i=1}^n \zeta_i$ with $\bzeta := (\zeta_1,\ldots,\zeta_n)$ a Rademacher vector. As seen in \cite[Example~1.3]{BFK17}, the equality 
$\res_0^{I_n}(\bpsi,\bxi) = \Theta(\sqrt{n})$ holds asymptotically almost surely. Noting that $\sr(I_n) = \omega(1)$, Corollary~\ref{cor:transparent} assertion for the sparse regime yields that the resilience of $I_n$ is a.a.s. $\Omega(\sqrt{n})$. This analysis extends to any diagonal matrix whose entries have the same order of magnitude. The aforementioned asymptotic tightness stipulated in Corollary~\ref{cor:transparent} is then established. 

\medskip
Allowing the entries of such matrices to have different orders of magnitudes, introduces diagonal matrices such as $D:= \mathrm{diag}(1,2^{-1},2^{-2},\ldots,2^{-(n-1)})$. Following~\cite[Example~1.5]{BFK17}, the resilience of the latter is $\Omega(n)$ asymptotically almost surely. Noting that $\sr(D) = O(1)$,  the matrix $D$ is not captured by our results for the bilinear and quadratic cases. Nevertheless, it is safe to say that the family of matrices $M$ satisfying $\sr(M) = \omega(1)$ is significantly richer than the set of matrices $M$ satisfying $\sr(M) = O(1)$.

\paragraph{Block-diagonal matrices}\label{par:block-matrix} The next member of the arc examined in terms of resilience estimations is the block-diagonal $n\times n$-matrix $M_w$ with block width $w$ satisfying $w \mid n$. For brevity, assume further that the latter is a $0/1$-matrix. As such, Corollary~\ref{cor:transparent} is made applicable for $M_w$ provided $w = o(n)$ holds; the latter is required in order to impose $\sr(M_w) = \omega(1)$.  Applying the aforementioned corollary yields that for any $x \in \mathbb{R}$ a.a.s., 
\begin{equation}\label{eq:block-matrix}
\res_x^{M_w}(\bpsi,\bxi) = \begin{cases}
    \Omega\left(\sqrt{\frac{n}{w}} \right), & 1 \leq w \leq n^{1/3},\\
    \Omega\left(\sqrt{\blw}(\frac n w)^{\frac 1 4}\right), & n^{1/3} < w \leq n/\log^2 n,\\
    \Omega\left(\sqrt{\frac{n}{\log n}}\right), & n / \log^2 n < w = o(n)
\end{cases}
\end{equation}
holds, where $\bpsi$ and $\bxi$ are per Corollary~\ref{cor:transparent}. The so called $\sqrt{n}$-limitation of our methods, mentioned in Remark~\ref{rem:limitation}, is illustrated here. 

\paragraph{Block-diagonal tensors}\label{para:block-diag-tensors}
At the end of our arc there lies the structure of block-diagonal tensor of arbitrary degree. In that, define the $d$-mode symmetric $\ell$-{\em scaled} tensor with {\em block-width} $\blw \in \mathbb{N}$ of dimensions $\underbrace{ n \times \dots \times n}_{d \textup{ times}}$, namely $(\fbl_{\bi})_{\bi\in\indset{d}}$, to be given by
$$
\fbl_{\bi} := 
\begin{cases}
    1, & \text{if}\; (a-1)\blw+1 \leq \bi \leq a \blw,\\
    0, & \text{otherwise},
\end{cases}
$$
whenever $a \in [n/\blw]$, where we tacitly assume that $\blw \mid n$. 
Applying Theorem~\ref{thm:main-poly} to $\fbl$ yields the following; proof of which is delegated to Appendix~\ref{par:block-diagonal}. 

\begin{claim}\label{clm:block-diagonal}
Let $d \geq 2$ and let $n$ be sufficiently large.
Then,
\begin{align*}
\sup_{x \in \mathbb{R}} & \PROBp{\res_x^{\fbl}(\bXi) \leq r} = \\ 
& 
O_d(1) \log^{\frac{d}{2}}(n) 
\sqrt{\frac{\blw}{n}}
\left(\left(1+\frac{r}{\sqrt{\blw}}\right)^d
-1\right)
  + 
  O_d(1)\exp \left(-\Omega_d(1) \left( \sqrt{\frac{n}{\blw}}\right)^{ 1/{\Theta_d(1)}} \right)    .
\end{align*}
Then, for $\blw = o(n)$ and any $x\in \mathbb{R}$
\begin{equation}\label{eq:block-diagonal-tensor}
\res_x^{\fbl}(\bxi_1,\dots,\bxi_d) = 
\Omega\left(\sqrt{\frac{\blw}{\log n}}\left(\frac{n}{\blw}\right)^{\frac 1 {2d}}\right)
\end{equation}
holds a.a.s., whenever $\bxi_1,\dots,\bxi_d \in \{\pm 1\}^n$ are independent Rademacher vectors. 
\end{claim}


\section{Model result: Resilience of bilinear Rademacher chaos}\label{sec:quad-bi}

In this section, we prove Theorem~\ref{thm::main-bf} which is served as a {\bf model result} through which we seek to enhance the understanding of the proof of our main result, namely Theorem~\ref{thm:main-poly}. 

Given $\ell \in \mathbb{N}$ and a Rademacher vector $\bmu$, write $N_\ell(\bmu)$ to denote the family of Rademacher vectors $\bzeta$ satisfying $d_H(\bmu,\bzeta) \leq \ell$.  
Performing a single flip in $\bpsi$ can alter the value of $\bpsi^\trans M \bxi$ by at most $2\norm{M\bxi}_\infty$. Performing a single flip in $\bxi$ may alter the value of $\norm{M\bxi}_\infty$ by at most $2\norm{M}_\infty$.
More generally, performing $\ell \in \mathbb{N}$ flips in $\bpsi$ and $\ell' \in \mathbb{N}$ flips in $\bxi$, such that $\ell+\ell' \leq r$, may alter the value of $\bpsi ^\trans M \bxi$ by at most
\begin{align*}
\max_{\bpsi'\in N_{\ell}(\bpsi)}  \max_{\bxi'\in N_{\ell'}(\bxi)} |(\bpsi') ^\trans M \bxi' - \bpsi ^\trans M \bxi|.  
\end{align*}
Noting that
\begin{align}
\label{eq:tria_ineq}
\left|(\bpsi') ^\trans M \bxi' - \bpsi ^\trans M \bxi\right|  
&=
\left|(\bpsi') ^\trans M \bxi' - 
\bpsi^\trans M \bxi'
+ 
\bpsi^\trans M \bxi'
- \bpsi ^\trans M \bxi
\right|
\\ \nonumber
& =
\left|(\bpsi'- \bpsi) ^\trans M \bxi' 
+ 
\bpsi^\trans M (\bxi'-\bxi)
\right|
\\ \nonumber
& \leq
\left|(\bpsi'- \bpsi) ^\trans M \bxi'\right| 
+ 
\left|\bpsi^\trans M (\bxi'-\bxi)
\right| 
\end{align}
allows us to write 
\begin{align}
\nonumber
\max_{\bpsi'\in N_{\ell}(\bpsi)}  \max_{\bxi'\in N_{\ell'}(\bxi)} |(\bpsi') ^\trans M \bxi' - \bpsi ^\trans M \bxi| 
&\leq \max_{\bpsi'\in N_{\ell}(\bpsi)}  \max_{\bxi'\in N_{\ell'}(\bxi)} 
 \left\{|(\bpsi'-\bpsi) ^\trans M \bxi'|
 + |\bpsi^\trans M (\bxi' - \bxi)| 
\right\}
 \\ \nonumber
& \leq 2\ell \max_{\bxi'\in N_{\ell'}(\bxi)} \norm{M\bxi'}_\infty  + 2 \ell' 
\norm{\bpsi^\trans M}_\infty 
\\ \nonumber
&\leq
2\ell \max_{\bxi'\in N_{\ell'}(\bxi)} \norm{M(\bxi'-\bxi)}_\infty  
+ 2 \ell \norm{M\bxi}_\infty 
+ 2 \ell' \norm{\bpsi^\trans M}_\infty 
\\ \nonumber
&\leq
4\ell\min\{\ell',\|M\|_{\infty,0}\}\norm{M}_\infty  
+ 2 \ell \norm{M\bxi}_\infty 
+ 2 \ell' \norm{\bpsi^\trans M}_\infty 
\\[4pt]
&\leq
4r\min\{r,\|M\|_{\infty,0}\}\norm{M}_\infty  
+ 2 r \norm{M\bxi}_\infty 
+ 2 r \norm{\bpsi^\trans M}_\infty.
\label{eq:decomp}
\end{align}
From the point of view of the rows of $M$, the minimisation appearing above takes care of the case where each row of $M$ has less than $\ell'$ non-zeros; in this case $\ell'- \norm{M}_{\infty,0}$ of the flips will have no effect on $\norm{M\bxi}_\infty$. 
The same applies to the columns. 
Repeating the same argument as above but with adding and subtracting $(\bpsi')^\trans M \bxi$ instead of $\bpsi^\trans M \bxi'$ in \eqref{eq:tria_ineq} yields the same inequality as in~\eqref{eq:decomp} but with $M$ replaced by $M^\trans$. We may thus take minimum between these two inequalities. This minimum is upper bounded by the sum of the following two quantities.
\begin{align}
\eps_M(\bxi, r)
&:= 
2r \|M\bxi\|_\infty
+ 2r \|M\|_\infty \min\{r,\|M\|_{\infty,0}\},
\nonumber\\ 
& \label{eq:eps_xi_k}\\
\eps_{M^\trans}(\bpsi, r) &:=
2r\|M^\trans \bpsi\|_\infty + 2r \|M\|_\infty\min\{r,\|M^\trans\|_{\infty,0}\}. \nonumber
\end{align}
The following inclusion of events 
\begin{align*}
\left\{\res_x^M(\bpsi,\bxi) \leq r\right\}
&\subseteq
\left\{|\bpsi ^\trans M \bxi - x | \leq \eps_M(\bxi,r) + \eps_{M^\trans}(\bpsi,r) \right\}
\\&\subseteq
\left\{|\bpsi ^\trans M \bxi - x | \leq 2\eps_M(\bxi,r)\right\}
\cup
\left\{|\bpsi ^\trans M \bxi - x | \leq 2\eps_{M^\trans}(\bpsi,r)\right\}
\end{align*}
then holds. We may thus write that
\begin{align}
\sup_{x\in\R}\PROBp{ \res_x^M(\bpsi,\bxi) \leq r} 
& \leq \sup_{x\in\R}\PROBp{ |\bpsi ^\trans M \bxi - x | \leq 
\eps_M(\bxi,r) + \eps_{M^\trans}(\bpsi,r)
}
\nonumber \\
\label{eq:small_ball_est_bf-1}
&\leq
\sup_{x\in\R}\PROBp{ |\bpsi ^\trans M \bxi - x | \leq 2\eps_M(\bxi,r)}
+
\sup_{x\in\R}\PROBp{ |\bpsi ^\trans M \bxi - x | \leq 2\eps_{M^\trans}(\bpsi,r)}
.
\end{align}

Lemma~\ref{lem:small_ball_bilinear}, stated next, is used to bound the small-ball probabilities seen on the right hand side of~\eqref{eq:small_ball_est_bf-1}. It is formulated in a slightly more general form; this in anticipation of our needs arising the proof of Theorem~\ref{thm:qf} below. A random variable $X$ is said to have the {\em lazy Rademacher} distribution if it assumes its values in the set $\{-1,0,1\}$ with the probabilities 
$$
\PROB\{X = 0\} = 1/2\quad \text{and} \quad \PROB\{X = 1\} = 1/4 = \PROB\{X = -1\}.
$$ 
A random vector $\bxi \in \{-1,0,1\}^n$ is said to be a {\em lazy Rademacher} vector provided its entries are i.i.d. copies of a lazy Rademacher random variable. 

\begin{lemma}
\label{lem:small_ball_bilinear}
Let $\bpsi$ be a real $n$-dimensional random vector and let $\bxi$ be a real $m$-dimensional random vector independent of $\bpsi$ and such that both are either Rademacher or both are lazy Rademacher. 

If $\PROBp{\eps_M(\bxi,r)\geq \norm{M\bxi}_\infty}=1$, then there are constants $c,C>0$ such that
\begin{align}
\label{eq:small_ball_bilinear}
\sup_{x\in\R}\PROBp{ |\bpsi ^\trans M \bxi - x | \leq \eps_M(\bxi,r)} 
\leq
\frac{c\EXPp{\eps_M(\bxi,r)}}{\norm{M}_\frob}
+ \exp \big(-C\sr(M)\big).
\end{align}    
\end{lemma}

Postponing the proof of Lemma~\ref{lem:small_ball_bilinear} until the end of this section, we proceed to deducing Theorem~\ref{thm::main-bf} from it. To see this, note that in the context of the latter, $\bxi$ appearing on the right hand side of \eqref{eq:small_ball_bilinear} is Rademacher (and not lazy). A trivial upper bound over the right hand side of \eqref{eq:small_ball_bilinear}, can be obtained through the inequalities $\norm{M\bxi}_\infty \leq \norm{M}_{\infty,0} \norm{M}_\infty$  and $\norm{\bpsi^\trans M}_\infty \leq \norm{M^\trans}_{\infty,0} \norm{M}_\infty$ applied to the terms $\|M\bxi\|_\infty$ and $\|\bpsi^\trans M\|_\infty$, respectively, appearing in~\eqref{eq:eps_xi_k}; indeed, with these estimates one may proceed to bound the expectations
of the quantities seen in~\eqref{eq:eps_xi_k}
by 
$$
\EXPp{\eps_M(\bxi,r)}\leq cr
\norm{M}_\infty
\norm{M}_{\infty,0}
\quad\text{and}\quad
\EXPp{\eps_{M^\trans}(\bpsi,r)}\leq cr
\norm{M}_\infty
\norm{M^{\trans}}_{\infty,0};
$$ 

Alas, $\|M\|_{\infty,0}$ and $\|M^\trans\|_{\infty,0}$ can be quite large and so a new idea is needed. Starting with the rows of $M$, a crucial observation here allowing us to improve on the above trivial upper bound is that regardless of whether $\bxi$ is Rademacher or lazy Rademacher, each random variables $X_{\bt}$ of the random process 
\begin{equation}\label{eq::gauss-proc}
\bX:=(X_{\bt}:=|\bt^\trans \bxi|)_{\bt \in \mathrm{Rows}(M)}
\end{equation}
is sub-gaussian\footnote{We follow~\cite[Definition~3.5]{van2014probability}. } with parameter $c\|\bt\|^2_2$, for some constant $c>0$, and thus sub-gaussian with parameter $c\|M\|^2_{\infty,2}$, for some constant $c>0$, as well; this, by Hoeffding's inequality (see e.g.~\cite[Theorem~2.2.5]{Vershynin}). All of the above applies to the columns of $M$ essentially verbatim. In that, we consider the random process 
$$
\bZ:=(Z_{\bt}:=|\bpsi^\trans \bt|)_{\bt \in \mathrm{Cols}(M)},
$$
where $\mathrm{Cols}(M)$ denotes the set of columns of $M$. This process is sub-gaussian with parameter $c\|M^\trans\|_{\infty,2}^2$ for some constant $c>0$. Consequently, each of the quantities 
$$
\EXPp{\sup_{\bt \in \mathrm{Rows}(M)}X_{\bt}} = \EXPp{\|M\bxi\|_\infty} \quad \text{and} \quad \EXPp{\sup_{\bt \in \mathrm{Cols}(M)}Z_{\bt}} = \EXPp{\|\bpsi^\trans M\|_\infty},
$$ 
which coincide with the so called {\em Rademacher complexity} of the sets $\mathrm{Rows}(M)$ and $\mathrm{Cols}(M)$, respectively, can each be estimated using the following special case of {\em Dudley's maximal inequality}.

\begin{theorem}\label{thm::dudley-max}{\em (Dudley's maximal inequality - abridged~\cite[Lemma~5.1]{van2014probability})}\\
Let $T \subseteq \R^n$ be finite and let $(Y_{\bt})_{\bt \in T}$ be a random process such that $Y_{\bt}$ is sub-gaussian with parameter $K^2$ for every $\bt \in T$. Then, 
$$
\EXPp{ \sup_{\bt \in T} Y_{\bt} } \leq K \sqrt{2 \log |T|}.
$$
\end{theorem}

For a Rademacher $\bxi$, as per the case of Theorem~\ref{thm::main-bf}, we may then write that 
\begin{equation}\label{eq:logn-appears}
\EXP\left\{\norm{M\bxi}_\infty\right\}\leq
c\sqrt{\log n}\norm{M}_{\infty,2}
\end{equation}
holds for some constant $c>0$. This in turn yields 
\begin{align}
\EXPp{\eps_M(\bxi,r)} & \leq    
c'_1 r\sqrt{\log n} \norm{M}_{\infty,2}
+c'_2 r \min\{r,\norm{M}_{\infty,0}\}  \norm{M}_\infty \nonumber \\
&\label{eq:exp-radius-bound} \\
\EXPp{\eps_{M^\trans}(\bpsi,r)}  &\leq c'_1 r\sqrt{\log n} \norm{M^\trans}_{\infty,2}
+c'_2 r \min\{r,\norm{M^\trans}_{\infty,0}\}  \norm{M}_\infty \nonumber
\end{align}
for some constants $c'_1,c'_2 >0$. Theorem~\ref{thm::dudley-max}, inequality~\eqref{eq:small_ball_est_bf-1}, and the facts that $\|M\|_\frob = \|M^\trans\|_\frob$ as well as $\sr(M) = \sr(M^\trans)$ allow for 
\begin{equation}\label{eq:sum-up}
\sup_{x \in \mathbb{R}}\PROBp{\res_x^M(\bpsi,\bxi) \leq r} \leq \\
\frac{c\big(\Ex \left\{\eps_M(\bxi,r)  \right\} + \Ex \left\{\eps_M(\bpsi,r) \right\}\big)}{\|M\|_\frob}  + \exp\big(-C\sr(M)\big)
\end{equation}
for some constants $c,C >0$. Owing to~\eqref{eq:exp-radius-bound}, we may write
\begin{align*}
\Ex \left\{\eps_M(\bxi,r)  \right\} & + \Ex \left\{\eps_M(\bpsi,r) \right\} \leq \\
&c'_1 r \sqrt{\log n} \max\{\|M\|_{\infty,2},\|M^\trans\|_{\infty,2}\}+ c'_2r \|M\|_\infty \min \{r,\max\{\|M\|_{\infty,0},\|M^\trans\|_{\infty,0}\}\}.
\end{align*}
This last bound together with~\eqref{eq:sum-up} 
establish Theorem~\ref{thm::main-bf}.

\medskip

It remains to prove Lemma~\ref{lem:small_ball_bilinear}. Facilitating our proof of the latter is the so called {\em Kolmogorov-Rogozin inequality} stated next. For a real random variable $X$ and $\eps >0$, the {\em L\'evy concentration function} is given by 
\begin{equation}\label{eq:Levy}
\levy_X(\eps) := \sup_{x \in \R} \PROB \left\{|X-x| \leq \eps \right\}.
\end{equation}

\begin{theorem}\label{thm:KRI}{\em (Kolmogorov-Rogozin inequality~\cite{kolmogorov1958proprietes,rogozin1961increase,rogozin1961estimate})}
Let $X_1,\ldots,X_n$ be independent random variables and set $S_n := \sum_{i=1}^n X_n$. Then, there exists a constant $C>0$ such that for any $0 \leq \eps_1,\ldots,\eps_n < \eps$,
$$
    \levy_{S_n}(\eps) \leq \frac{C\eps}{\sqrt{\sum_{i=1}^n \eps_i^2(1 - \levy_{X_i}(\eps_i))}}
$$
holds. 
\end{theorem}

It remains to prove Lemma~\ref{lem:small_ball_bilinear}.

\bigskip
\noindent
{\bf Proof of Lemma \ref{lem:small_ball_bilinear}.}
Start by writing
\begin{align}
\sup_{x\in\R}&\PROBp{ |\bpsi ^\trans M \bxi - x | \leq \eps_M(\bxi,r)} 
\leq \nonumber \\
& \EXP_{\bxi}\left[\sup_{x\in\R}\PROBp{ |\bpsi ^\trans M \bxi - x | \leq \eps_M(\bxi,r)
\,\, \wedge \,\, 
\norm{M\bxi}_2 \geq \frac{1}{2} \norm{M}_\frob
\,\middle\vert\,\bxi}\right]
+ \PROBp{\norm{M\bxi}_2 < \frac{1}{2}\norm{M}_\frob}. \label{eq:conditioning-1}   
\end{align}
By the Hanson-Wright inequality~\cite[Theorem 2.1]{rudelson2013hanson}, there exists a constant $C>0$ such that 
\begin{equation}\label{eq::below-frob}
\PROBp{\|M\bxi\|_2 < \frac{1}{2}\norm{M}_\frob} \leq \exp{\left(-C\sr(M)\right)}.
\end{equation}
To bound the term appearing within the expectation appearing on the right hand side of~\eqref{eq:conditioning-1}, we appeal to  Theorem~\ref{thm:KRI}. To that end, fix $\bxi$ satisfying $\norm{M\bxi}_2 \geq \frac{1}{2} \norm{M}_\frob >0$ (the last inequality is owing to $M$ being non-zero) and set (per the terminology of Theorem~\ref{thm:KRI}), 
\begin{align*}
 X_i = \bpsi_i(M\bxi)_i, 
  \qquad 
 \eps =  \eps_M(\bxi,k),
 \qquad \eps_i = |(M \bxi)_i|/2,
\end{align*}
where $\eps_M(\bxi,k)$ is as in \eqref{eq:eps_xi_k} and note that $\bpsi^\trans M \bxi  = \sum_{i=1}^n X_i$. Owing to $\norm{M\bxi}_2 >0$, at least one $\eps_i$ is non-zero. For each such $i$, if $\bxi$ is Rademacher, then all the mass of $X_i$ is supported on two values, i.e. $\pm |(M\bxi)_i|$; if $\bxi$ is a lazy Rademacher vector then all the mass of $X_i$ is supported on the set $\{-|(M\bxi)_i,0,|(M\bxi)_i|\}$. In either case, $\levy_{X_i}(\eps_i) = 1/2$. 
We may then write 
\begin{align*}
\sum_{i=1}^n \eps_i^2(1 - \levy_{\eps_i}(X_i)) 
=
\frac{1}{8}\sum_{i=1}^n (M\bxi)_i^2 = \frac{1}{8} \norm{M\bxi}_2^2.
\end{align*}
Theorem \ref{thm:KRI} then asserts that 
\begin{align*}
\sup_{x\in\R}\PROBp{ |\bpsi ^\trans M \bxi - x | 
\leq \eps_M(\bxi,r)
\,\, \wedge \,\, 
\norm{M\bxi}_2 \geq \frac{1}{2}\norm{M}_\frob
\,\,\middle\vert\,\, \bxi} 
&\leq 
\frac{4\eps_M(\bxi,r)}{\norm{M\bxi}_2}\cdot\chr{\norm{M\bxi}_2 \geq \frac{1}{2}\norm{M}_\frob}
\\&\leq 
\frac{8\eps_M(\bxi,r)}{\norm{M}_\frob}.
\end{align*}
Substituting this bound as well as~\eqref{eq::below-frob} into \eqref{eq:conditioning-1}, one obtains
\begin{align*}
\sup_{x\in\R}\PROBp{ |\bpsi ^\trans M \bxi - x | \leq \eps_M(\bxi,r)}
& \leq \frac{c\EXP_{\bxi}\left\{\eps_M(\bxi,r)\right\}}{\norm{M}_\frob}
+ \exp{\left(-C\sr(M)\right)},
\end{align*}
concluding the proof of Lemma \ref{lem:small_ball_bilinear}. \endProof

\section{Resilience of Rademacher chaos of arbitrary degree}\label{sec:main-poly}

In this section, we prove Theorem~\ref{thm:main-poly}. In fact, we prove the following more accurate version of it. 

\begin{theorem}\label{thm:main-poly-2}
Let $2 \leq d \in \mathbb{N}$ be given an let $(f_{\bi})_{\bi \in \mathcal{T}}$ be as in~\eqref{eq:poly-def} such that $f \not\equiv 0$ and satisfying 
$n_i= n = \omega(d^d)$ for every $i \in [d]$. Let $\bxi_1,\ldots,\bxi_d$ be independent Rademacher vectors conformal with the dimensions of $f$. Then,  
{\small
\begin{align*}
&\sup_{x\in\R}\PROBp{ \res_x^f(\bXi) \leq \ress}
\leq\\
&
\sum_{I \subseteq [d]: I\neq\emptyset}
\frac{O_d((2\ress)^{|I|})
    \cdot
     \log^{\frac{d}{2}}\left(\Pi_{i\in I} n\right)
     \cdot\max_{d_I \in \mathcal{T}_I}\norm{f_{I,d_I}}_\frob}{ \|f\|_\frob}
   + \sum_{I \subseteq [d]: I\neq\emptyset} \min_{i \in [I]} \exp \left(-\Omega_d(1) \min_{\substack{k \in [2(d-1)]  \\ k\; \text{even}}} \left( \frac{ \|f\|^2_\frob}{\GameNorm}\right)^{1/\Theta_d(1)} \right)
\end{align*}
}
holds whenever $\ress \in [n]$. 
\end{theorem}

\noindent
Proof of Theorem~\ref{thm:main-poly-2} can be found in Section~\ref{sec:proof-main-poly}; prior to this, we collect additional notation and results facilitating our proof of this result in Section~\ref{sec:pre-poly}. 

\subsection{Preliminaries}\label{sec:pre-poly}

\subsubsection{Vectors of chaos restrictions}\label{sec:tensor-restrictions}
Given $f$ per the premise of Theorem~\ref{thm:main-poly-2}, we are reminded of the selection function, namely $\delta_{I,d_I}$, set in~\eqref{eq:selection} and the notion of chaos restrictions $f_{I,d_I}$ set in~\eqref{eq:restriction}. Examples facilitating the  understanding of these two notions are provided next.

\begin{example}
{\em Consider the degree~$3$ chaos given by  
$$
g(\bxi_1, \bxi_2,\bxi_3)= \sum_{i \in [n_1]} \sum_{j \in [n_2]} \sum_{k \in [n_3]} g_{ijk} (\bxi_1)_i (\bxi_2)_j(\bxi_3)_k.
$$ 
Pick, say, $I = \{2,3\}$ and let $d_I := (d_2 = j' \in [n_2], d_3 = k' \in [n_3])$. Then, 
\begin{align*}
g_{I,d_{I}}(\bxi_1)  = g(\bxi_1, \be_{d_2}, \be_{d_3}) & = \sum_{i \in [n_1]} \sum_{j \in [n_2]} \sum_{k \in [n_3]} g_{ijk} (\bxi_1)_i (\be_{d_2})_j (\be_{d_3})_k\\
& = \sum_{i \in [n_1]} \sum_{j \in [n_2]} \sum_{k \in [n_3]} g_{ijk} (\bxi_1)_i (\be_{j'})_j (\be_{k'})_k\\
& = \sum_{i \in [n_1]} f_{ij'k'} (\bxi_1)_i.
\end{align*}
\hfill$\blackdiamond$
}
\end{example}

\begin{example}
{\em Individual coefficients of $(f_{\bi})_{\bi \in \mathcal{T}}$ are isolated through 
$
f(\be_{i_1},\ldots,\be_{i_d}) = f_{i_1i_2\cdots i_d}.
$}
\hfill$\blackdiamond$    
\end{example}

Given a nonempty set of dimensions $I \subseteq [d]$, the maximum magnitude of all restrictions possible along these dimensions across all directions is given by 
\begin{equation}\label{eq:res-inf-norm}
\|f_{I}(\bXi_{\quot{I}})\|_{\infty} := \max_{d_I \in \mathcal{T}_I} |f_{I,d_I}(\bXi_{\quot{I}}))|.
\end{equation}
Unlike the norms $\|f_{I,d_I}\|_\frob$ and $\|f_{I,d_I}\|_\infty$, the term $\|f_{I}(\bXi_{\quot{I}})\|_{\infty}$ is a random variable.

Given a dimension $i \in [d]$, define the (random) vector
$$
\bv_{f,i} (\bXi_{\quot{i}}):= \big(f_{\{i\},\{k\}}(\bXi_{\quot{i}}): k \in [n_i] \big)^\trans \in \mathbb{R}^{n_i}
$$
to be a vector of chaos restrictions, 
where we recall that 
$$
f_{\{i\},\{k\}}(\bXi_{\quot{i}}) = f(\bxi_1,\ldots,\bxi_{i-1},\be_k,\bxi_{i+1}, \ldots, \bxi_d).
$$
Aiding the reader to locate $\bv_{f,i} (\bXi_{\quot{i}})$ in the proof of Theorem~\ref{thm::main-bf}, note that the counterparts of $\bv_{f,i}(\bXi_{\quot{i}})$ in that proof are $M\bxi$ and $\bpsi^\trans M$. Recalling the matrix $\mathcal{A}_i$ defined in~\eqref{eq:order-coeff}, note that 
\begin{equation}\label{eq:op-to-vec}
\mathcal{A}_i \big( \mathrm{vec} (\bXi_{\quot{i}}) \big) = \bv_{f,i}(\bXi_{\quot{i}})
\end{equation}
holds provided, of course, that the vectorisation of the (decoupled Rademacher) tensor $\bXi_{\quot{i}}$ conforms with the ordering of the coefficients of $f$ throughout $\mathcal{A}_i$ as defined in~\eqref{eq:order-coeff}; an ordering we assume is upheld.

\subsubsection{Concentration of Rademacher tensors} \label{sec:tensor-concentrate}
In the proof of Theorem~\ref{thm::main-bf}, utilization of the Hanson-Wright inequality~\cite[Theorem 2.1]{rudelson2013hanson} can be seen in~\eqref{eq::below-frob}. In the context of the proof Theorem~\ref{thm:main-poly-2}, the counterpart of~\eqref{eq::below-frob} is Lemma~\ref{lem:v-concentration} stated next; the lemma asserts that for every $i \in [d]$, the random variable $\|\bv_{i,f}(\bXi_{\quot{i}})\|_2$ exhibits an adequate level of concentration around $\|f\|_\frob$.

\begin{lemma}\label{lem:v-concentration}
For every $i \in [d]$, 
$$
\PROBp{\|\bv_{i,f}(\bXi_{\quot{i}})\|_2 \leq \frac{\|f\|_\frob}{2}} \leq \exp \left(-\Omega_d(1) \min_{\substack{ k \in [2(d-1)]  \\ k\; \text{even}}} \left( \frac{ \|f\|^2_\frob}{\GameNorm}\right)^{1/\Theta_d(1)} \right).
$$
\end{lemma}

To prove Lemma~\ref{lem:v-concentration}, we employ a far reaching generalisation of the Hanson-Wright inequality put forth by Adamczak and Wolff~\cite[Theorem~1.4]{AW15}. Statement of the latter requires preparation. Given a degree $D\geq 1$ polynomial $g :\R^n \to \R$ and $m \in [D]$, write $\nabla^m g$ to denote the tensor of $m$th-derivatives given by 
$$
(\nabla^m g)_{i_1\cdots i_d} = \frac{\partial g}{\partial x_{i_1} \cdots \partial x_{i_m}}.
$$
Derivation being insensitive to the order of the derivation sequence $\partial x_{i_1} \cdots \partial x_{i_m}$ means that the tensor $\nabla^mg$ retains $m!$ copies of each $m$th-derivative. 

The following is an abridged and significantly weaker formulation of~\cite[Theorem~1.4]{AW15} fitted for Rademacher vectors and to our needs. 

\begin{theorem}\label{thm:Adamcask} 
Let $g :\R^n \to \R$ be a polynomial of degree $D \geq 1$ and let $\bxi \in \{\pm 1\}^n$ be a Rademacher vector. Then, 
\begin{equation}\label{eq:Adamcask}
\PR \left\{|g(\bxi) - \Ex g(\bxi)| \geq t\right\} \leq 2 \exp \left(-\Omega_D(1) \min_{m \in [D]}  \left(\frac{t}{\|\Ex \nabla^m g(\bxi)\|_\frob} \right)^{1/\Theta_D(1)} \right)
\end{equation}
holds every $t > 0$. 
\end{theorem}

We are now in position to prove Lemma~\ref{lem:v-concentration}. 

\medskip
\noindent
{\bf Proof of Lemma~\ref{lem:v-concentration}.} It suffices to prove that 
\begin{equation}\label{eq:v-tail}
\PROBp{\|\bv_{i,f}(\bXi_{\quot{i}})\|_2^2 \leq \frac{\|f\|^2_{\frob}}{4}} \leq \exp \left(-\Omega_d(1) \min_{\substack{k \in  [2(d-1)]  \\ k\; \text{even}}} \left( \frac{ \|f\|^2_\frob}{\GameNorm}\right)^{1/\Theta_d(1)} \right). 
\end{equation}
Write
\begin{equation}\label{eq:v-poly}
\|\bv_{i,f}(\bXi_{\quot{i}})\|_2^2  = \mathrm{vec}(\bXi_{\quot{i}})^\trans \mathcal{A}_i^\trans \mathcal{A}_i\mathrm{vec}(\bXi_{\quot{i}}) = \sum_{\bj \in \indset{d-1}} \sum_{\bk \in \indset{d-1}} (\mathcal{A}_i^\trans\mathcal{A}_i)_{\bj \bk} (\bXi_{\quot{i}})_{\bj} (\bXi_{\quot{i}})_{\bk};
\end{equation}
aim is then to apply Theorem~\ref{thm:Adamcask} to the Rademacher polynomial~\eqref{eq:v-poly} in order to establish~\eqref{eq:v-tail}

\bigskip
\noindent
{\tt Expectation of~\eqref{eq:v-poly}.} To ascertain the expectation of the Rademacher polynomial~\eqref{eq:v-poly}, write the latter as follows
\begin{equation}\label{eq:v-decomposition}
\sum_{\substack{\bj,\bk \in \indset{d-1} \\ \bj \neq \bk}}  (\mathcal{A}_i^\trans\mathcal{A}_i)_{\bj \bk} (\bXi_{\quot{i}})_{\bj} (\bXi_{\quot{i}})_{\bk} + \sum_{\bj \in \indset{d-1}} (\mathcal{A}_i^\trans\mathcal{A}_i)_{\bj \bj}
\end{equation}
and observe the identities 
$$
\sum_{\bj \in \indset{d-1}} (\mathcal{A}_i^\trans\mathcal{A}_i)_{\bj \bj} = \|f\|^2_\frob \quad \text{as well as}\quad \Ex \left\{ \sum_{\substack{\bj,\bk \in \indset{d-1} \\ \bj \neq \bk}}  (\mathcal{A}_i^\trans\mathcal{A}_i)_{\bj \bk} (\bXi_{\quot{i}})_{\bj} (\bXi_{\quot{i}})_{\bk}\right\} = 0.
$$
We may thus write 
$$
\Ex \left\{\sum_{\bj \in \indset{d-1}} \sum_{\bk \in \indset{d-1}} (\mathcal{A}_i^\trans\mathcal{A}_i)_{\bj \bk} (\bXi_{\quot{i}})_{\bj} (\bXi_{\quot{i}})_{\bk} \right\} = \|f\|^2_{\frob}.
$$

\bigskip
\noindent
{\tt Partial derivatives of~\eqref{eq:v-poly}.} In order to simplify notation, assume,  without loss of generality, that $i=d$ and write $\mathcal A:=\mathcal A_d$; analysis for $i\neq d$ follows by symmetry. For the real vector variables $\bx_1,\dots,\bx_{d-1}\in\R^{n}$ and a tuple $\bell=\big((b_1,c_1),\dots,(b_m,c_m)\big)\in([d-1]\times[n])^m$, define the real monomial 
$$
\bx[\bell] = 
(\bx_{b_1})_{c_1} \cdots (\bx_{b_m})_{c_m}.
$$
This notation, we also use for Rademacher monomials, written $\bxi[\bell]$, where here real variables of the form $(\bx_b)_c$ are replaced with corresponding entries of the Rademacher vectors $\bxi_1,\dots,\bxi_{d-1}\in\{\pm1\}^{n}$.

Let $\bij:[n]^{d-1} \to [d-1]\times[n]^{d-1}$ be such that for any $\bj\in[n]^{d-1}$, $$\bij(\bj):=\left((1,\bj_1),\dots,(d-1,\bj_{d-1})\right).$$
The latter allows us to consider the real polynomial corresponding to \eqref{eq:v-poly},
\begin{align*}
g(\bx_1,\dots,\bx_{d-1})
: =
\sum_{\bj,\bk \in \indset{d-1}}
(\mathcal{A}^\trans\mathcal{A})_{\bj \bk} \cdot
\bx[\bij(\bj)]\cdot \bx[\bij(\bk)].
\end{align*}

Aim is to determine the $m$th-derivatives of $g$,  
$$
\bm\nabla^m g(\bx_1,\dots,\bx_{d-1})
\in \underbrace{\R^{(d-1)\times n}\times\dots\times \R^{(d-1)\times n}}_{m\text{ times}},
$$
whenever $m \in [2(d-1)]$. In that, $g$ is viewed as a of degree $2(d-1)$ polynomial in the $(d-1)n$ variables $(\bx_1)_1,
\dots,(\bx_{d-1})_{d-1}$.
For $\bell=\big((b_1,c_1),\dots,(b_m,c_m)\big)\in([d-1]\times[n])^m$, the $\bell$th entry of $\bm\nabla^m g$ is denoted as
\begin{align*}
\bm\nabla^m_{\bell} g(\bx_1,\dots,\bx_{d-1}) := (\bm\nabla^m g(\bx_1,\dots,\bx_{d-1}))_{\bell}
= 
\frac{\partial^m g(\bx_1,\dots,\bx_{d-1}) } {\partial (\bx_{b_1})_{c_1}\cdots\partial (\bx_{b_m})_{c_m}}.    
\end{align*}
Linearity of differentiation then allows for
\begin{align}
\label{eq:diff_real_poly}
\bm\nabla^m g(\bx_1,\dots,\bx_{d-1})
=
\sum_{\bj,\bk \in \indset{d-1}} 
(\mathcal{A}^\trans\mathcal{A})_{\bj \bk} \cdot
\bm\nabla^m (
\bx[\bij(\bj)]
\bx[\bij(\bk)]).
\end{align}

Since $f$ is homogenous and multilinear,
the contributing monomials $\bx[\bij(\bj)]\bx[\bij(\bk)]$ in \eqref{eq:diff_real_poly} have a simple structure.
First, any variable $(\bx_b)_c$ from the  $(d-1)n$ variables may appear only with a degree in $\{0,1,2\}$. In addition, those variables appearing with degree $2$ are reached only through a matched pair of indices $\bj_\ell=\bk_\ell$ for some $\ell\in[d-1]$.
An additional source of simplification stems from the fact that Rademacher random variables $\xi\in\{\pm1\}$ are to be substituted into the variables of $g$ and the latter always satisfy $\xi^2=1$.  
Such squares can then be eliminated (i.e., replaced with $1$) in 
$\bx[\bij(\bj)]\bx[\bij(\bk)]$ without affecting the distribution of $g(\bxi_1,\dots,\bxi_{d-1})$.
Put another way, the replacement
\begin{align*}
\bx[\bij(\bj)]\bx[\bij(\bk)] \,\to\,    \bx[(\bij(\bj)\cup\bij(\bk))\setminus(\bij(\bj)\cap\bij(\bk))]
\end{align*}
can be applied.
It follows that for any $(b,c)\in\bij(\bj)\cap\bij(\bk)$,
\begin{align*}
\EXPp{\frac
{\partial(\bxi[\bij(\bj)]\bxi[\bij(\bk)])
}{\partial (\bxi_b)
_{c}}}
=\EXPp{\frac
{\partial\bxi[(\bij(\bj)\cup\bij(\bk))\setminus(\bij(\bj)\cap\bij(\bk))]
}{\partial (\bxi_b)
_{c}}} = 0.
\end{align*}

Post the aforementioned elimination of squares, the remaining monomials have the form 
$$
\bx[(\bij(\bj)\cup\bij(\bk))\setminus(\bij(\bj)\cap\bij(\bk))]
= \bx[\bij(\bj)\setminus(\bij(\bj)\cap\bij(\bk))]\cdot\bx[\bij(\bk)\setminus(\bij(\bj)\cap\bij(\bk))];
$$
these are of even degree $2(d-1)-2|\bij(\bj)\cap\bij(\bk)|$ and contain only distinct variables all of degree one, such that $|\bij(\bj)\setminus(\bij(\bj)\cap\bij(\bk))|=|\bij(\bk)\setminus(\bij(\bj)\cap\bij(\bk))|$.
In particular, if such a leftover monomial, so to speak,  
contains a variable which differentiation is not carried out with respect to it, then the differentiated monomial, if not vanished, still contains that variable with degree one. Hence,  
its expectation, when each $\bx$ is replaced with its corresponding Rademacher $\bxi$, is zero.
This also implies that for any odd $m$,
\begin{align*}
    \EXPp{
\bm\nabla^m(\bxi[\bij(\bj)]\bxi[\bij(\bk)])
}=0.
\end{align*}
It follows that for any $\bell=\big((b_1,c_1),\dots,(b_m,c_m)\big)\in([d-1]\times[n])^m$
of even $m$, 
\begin{align*}
\EXPp{\bm\nabla^m_{\bell} (\bxi[\bij(\bj)]\bxi[\bij(\bk)])}  
=
    \begin{cases}
        1, &  (\bj,\bk)\in \Ext_m(\bell),
        \\
        0, & \text{otherwise} ,
    \end{cases}
\end{align*}
where we recall that
\begin{align*}
\Ext_m(\bell):= \Ext_{d,m}(\bell) = \left\{(\bj,\bk)\in [n]^{d-1}\times[n]^{d-1}:
(\bij(\bj)\cup\bij(\bk))\setminus(\bij(\bj)\cap\bij(\bk))=\bell
\right\}
\end{align*}
from~\eqref{eq:ext}. Conceptually, it is conducive to think of the set $\Ext_m(\bell)$ as comprised of the monomials that the tuple $\bell$ can be extended to (through $\bij$). In that, $\Ext_m(\bell)=\emptyset$ may hold for numerous tuples $\bell\in([d-1]\times[n])^m$.
Indeed, $D_m(\bell)\neq\emptyset$ if and only if there exist $\bj,\bk\in[n]^{d-1}$ such that $(\bij(\bj)\cup\bij(\bk))\setminus(\bij(\bj)\cap\bij(\bk))=\bell$.
This may happen only when $\bell=\big((b_1,c_1),\dots,(b_m,c_m)\big)$ is such that $(b_i,c_i)$ are all distinct and the value of any $b_i$ appears in $\bell$ exactly twice. It is then conducive to define
$\Rel_m := \Rel_{d,m}$  which has been introduced already in~\eqref{eq:Rel}. 

\bigskip
\noindent
{\tt Frobenius norm estimation.} Gearing up towards an application of Theorem~\ref{thm:Adamcask}, we next present  
an estimation for $\norm{\EXPp{\bm\nabla^m g(\bxi_1,\dots,\bxi_{d-1})}}_{\frob}^2$, whenever $m\in[2(d-1)]$ is even.
\begin{align}
\norm{\EXPp{\bm\nabla^m g(\bxi_1,\dots,\bxi_{d-1})}}_{\frob}^2
&=
\sum_{\bell\in([d-1]\times[n])^m}
\left(\EXPp{\bm\nabla^m g(\bxi_1,\dots,\bxi_{d-1})}\right)_{\bell}^2 \nonumber
\\&=
\sum_{\bell\in([d-1]\times[n])^m}
\left(\EXPp{\bm\nabla^m_{\bell} g(\bxi_1,\dots,\bxi_{d-1})}\right)^2 \nonumber
\\&=
\sum_{\bell\in([d-1]\times[n])^m}
\left(
\sum_{\bj,\bk \in \indset{d-1}} 
(\mathcal{A}^\trans\mathcal{A})_{\bj \bk} \cdot
\EXPp{\bm\nabla^m_{\bell} (
\bxi[\bij(\bj)]
\bxi[\bij(\bk)])}\right)^2 \nonumber
\\&=
\sum_{\bell\in([d-1]\times[n])^m}
\left(
\sum_{(\bj,\bk) \in \Ext_{m}(\bell)} 
(\mathcal{A}^\trans\mathcal{A})_{\bj \bk}
\right)^2 \nonumber
\\&=
\sum_{\bell\in \cL_m}
\left(
\sum_{(\bj,\bk) \in \Ext_{m}(\bell)} 
(\mathcal{A}^\trans\mathcal{A})_{\bj \bk}
\right)^2 =  \|\bm{\Game}_{d,m}f\|^2, \label{eq:GameNorm}
\end{align}
where for the last equality we recall that $i=d$ has been set above for comfort. 

\medskip
\noindent
{\tt Applying Theorem~\ref{thm:Adamcask}.} Write 
\begin{align*}
\PR \left\{ g(\bxi_1,\ldots\bxi_d) \leq \frac{\|f\|^2_\frob}{4} \right\} &  \overset{\phantom{\eqref{eq:GameNorm}}}{\leq}  \exp \left(-\Omega_d(1) \min_{\substack{k \in [2(d-1)]\\ k\; \text{even}}}  \left(\frac{\|f\|^2_\frob}{\|\Ex \nabla^k g(\bxi_1,\ldots,\bxi_d)\|_\frob} \right)^{1/\Theta_d(1)} \right) \\
& \overset{\eqref{eq:GameNorm}}{=} \exp \left(-\Omega_d(1) \min_{\substack{k \in [2(d-1)]\\ k\; \text{even}}}  \left(\frac{\|f\|^2_\frob}{\GameNorm} \right)^{1/\Theta_d(1)} \right). 
\end{align*}
This concludes our proof of Lemma~\ref{lem:v-concentration}.\hfill$\blacksquare$

\subsection{Proof of Theorem~\ref{thm:main-poly-2}}\label{sec:proof-main-poly}

Given $r\in\mathbb N$, then for a decoupled Rademacher chaos per the premise of Theorem~\ref{thm:main-poly-2}, performing 
$\ell_1,\dots,\ell_d$ flips on $\bxi_1,
\ldots,\bxi_d$, respectively, such that $\sum_{i=1}^d \ell_i \leq r$, may alter the value of $f(\bxi_1,
\ldots,\bxi_d)$ by at most
\begin{align} 
\label{eq:eps_def}
\Delta_{f,\bXi,r} := \Delta_{f,\bXi} (\ell_1,\ldots,\ell_d) := \maxoverball{\bpsi}{\bxi}{i}{d} \Big| f(\bvs{\bpsi}{1},
\dots, \bvs{\bpsi}{d}) - f(\bvs{\bxi}{1}, 
\dots, \bvs{\bxi}{d}) \Big|,
\end{align}
where, as defined above,
$
N_{\ell}(\bxi_i) := \{\bpsi\in\{-1,1\}^{n}: d_{\ham}(\bpsi,\bxi_i)\leq \ell\}
$
is the Hamming neighbourhood/ball of radius $\ell$ about $\bxi_i$. 

As seen in the proof of Theorem~\ref{thm::main-bf}, interest in $\Delta_{f,\bXi,k}$ arises from the following inclusion of events
\begin{equation}\label{eq:inclusion-poly-1}
\left\{\res_x^f(\bXi) \leq r\right\}
\subseteq
\Big\{|f(\bXi) - x | \leq \Delta_{f,\bXi,r}\Big\}.
\end{equation}
The following lemma establishes an upper bound on $\Delta_{f,\bXi,r}$; its counterpart in the proof of Theorem~\ref{thm::main-bf} is~\eqref{eq:eps_xi_k}.

\begin{lemma}\label{lem:poly-radius}
$
    \Delta_{f,\bXi} (\ell_1,\ldots,\ell_d) \leq
    \sum_{I \subseteq [d]: I\neq\emptyset} (2r)^{|I|}
    \cdot \|f_{I}(\bXi_{\quot{I}})\|_{\infty} =:\, \eps_{f,r}(\bXi_{\quot{I}})
$.
\end{lemma}

\noindent
We postpone the proof of Lemma~\ref{lem:poly-radius} until Section~\ref{sec:poly-raidus} and proceed with our proof of Theorem~\ref{thm:main-poly-2} assuming the former holds true. 

\medskip

Equipped with Lemma~\ref{lem:poly-radius}, we may proceed to note that the inclusion of events seen in~\eqref{eq:inclusion-poly-1} can be extended as to read as follows
\begin{align*}
\left\{\res_x^f(\bXi) \leq r\right\}
&\subseteq
\Big\{|f(\bXi) - x | \leq \eps_{f,r}(\bXi)\Big\}
\\ &\subseteq \bigcup_{I\subseteq[d]:I\neq\emptyset}\left\{
|f(\bXi) - x | \leq 2^d(2r)^{|I|} 
   \|f_{I}(\bXi_{\quot{I}})\|_{\infty}\right\},
\end{align*}
where the last inclusion is owing to the union-bound. We may then write
\begin{align}
\nonumber
\sup_{x\in\R}\PROBp{ \res_x^f(\bXi) \leq r}
& \leq \sup_{x\in\R}\PR\Big\{ |f(\bXi) - x | \leq \eps_{f,r}(\bXi)\Big\}
\\
\nonumber
&
\leq\sum_{I \subseteq [d]: I\neq\emptyset}
\sup_{x\in\R}\PR \Big\{ |f(\bXi)- x | \leq 2^d(2r)^{|I|} \cdot \|f_{I}(\bXi_{\quot{I}})\|_{\infty}\Big\}
\\
\label{eq:small_ball_est_bf}
&=
\sum_{I \subseteq [d]: I\neq\emptyset}
\sup_{x\in\R}\PR \Big\{ 
|f(\bXi)
- x | 
\leq \eps_{I,f,r}(\bXi_{\quot{I}})\Big\},
\end{align}
where
\begin{equation}
\label{eq:radius-poly-I}
\eps_{I,f,r}(\bXi_{\quot{I}}) :=2^d(2r)^{|I|}\cdot \|f_{I}(\bXi_{\quot{I}})\|_{\infty}.
\end{equation}

The next lemma bounds a single summand of the sum appearing on the right hand side of~\eqref{eq:small_ball_est_bf}; its counterpart in the proof of Theorem~\ref{thm::main-bf} is Lemma~\ref{lem:small_ball_bilinear}.

\begin{lemma}\label{lem:small-ball-multilinear}
There exist constant $C > 0 $ (independent of $d$), such that
for any nonempty $I\subseteq [d]$ and 
$r \in [n]$,
\begin{align}
\sup_{x\in\R}\PR \Big\{ & |f(\bXi)- x | \leq 
\eps_{I,f,r}(\bXi_{\quot{I}})
\Big\}
\leq \nonumber  \\
& \frac{C\EXPp{\eps_{I,f,r}(\bXi_{\quot{I}})}}{\|f\|_\frob} + \min{i \in I}\exp \left(-\Omega_d(1) \min_{\substack{ k \in [2(d-1)]  \\ k\; \text{even}}} \left( \frac{ \|f\|^2_\frob}{\GameNorm}\right)^{1/\Theta_d(1)} \right). \label{eq:small_ball_poly}
\end{align}    
\end{lemma}

Postponing the proof of Lemma~\ref{lem:small-ball-multilinear} until Section~\ref{sec:lem:small-ball-multilinear}, we proceed with our argument for Theorem~\ref{thm:main-poly}. The next ingredient is an upper bound on $\EXPp{\eps_{I,f,r}(\bXi)}$ seen on the right hand side of~\eqref{eq:small_ball_poly}. The following lemma delivers such a bound; statement of which requires 
that a specific version (taken from~\cite[Equation~(4.3.2)]{de2012decoupling}) of the so called {\sl Young modulus function} be defined. To that end, set $\alpha = 2/d$, put $x_0 =\left(\frac{1-\alpha}{\alpha}\right)^{\frac{1}{\alpha}}$, and define $\YM:\R^+ \to \R^+$ to be given by 
\begin{align}
\label{eq:Phi}
    \YM(x) =
    \begin{cases}
        (1-\alpha)e^{x_0^\alpha} \left(\frac{x}{x_0}\right), & \quad 0\leq x < x_0;
        \\ e^{x^\alpha}-\alpha e^{x_0^\alpha}, & \quad x_0 \leq x.
    \end{cases}
\end{align}
As noted in~\cite{de2012decoupling}, $\Phi$ is strictly increasing to $\infty$, convex, and $\Phi(0)=0$ so that $\Phi^{-1}$ is well defined and concave. The counterpart to the bound obtained in the proof of Theorem~\ref{thm::main-bf} using Dudley's maximal inequality (Theorem~\ref{thm::dudley-max}), namely~\eqref{eq:logn-appears}, reads as follows.

\begin{lemma}\label{lem:bound-expectation}
For any nonempty set of dimensions $I \subseteq [d]$, the equality 
\begin{equation}\label{eq:bound-expectation}
\EXPp{\eps_{I,f,r}(\bXi_{\quot{I}})} = O_d(1) (2r)^{|I|}
    \cdot \max_{d_I \in \mathcal{T}_I}\norm{f_{I,d_I}}_\frob
    \cdot
    \Phi^{-1}\left(\Pi_{i\in I} n\right)
\end{equation}
holds.
\end{lemma}

\noindent
The quantity $O_d(1)$ seen in~\eqref{eq:bound-expectation} is exponential in $d$. 
A proof of Lemma~\ref{lem:bound-expectation} can be seen in Section~\ref{sec:lem:bound-expectation}. 

\medskip

The assumption that $n = \omega(d^d)$ for every $i \in [d]$, appearing in the premise of Theorem~\ref{thm:main-poly-2}, implies that  $n^{|I|} = \Pi_{i\in I} n\geq x_0$ holds (as $x_0 =O(d^d)$, by definition) and thus, on account of $\Phi$ being strictly increasing, $\Phi(n^{|I|}) \geq \Phi(x_0)$ holds as well. The latter, coupled with~\eqref{eq:bound-expectation} and the definition of $\Phi^{-1}(\cdot)$, yields that for any nonempty $I \subseteq [d]$ the following equality 
\begin{equation}\label{eq:bound-expectation-2}
\EXPp{\eps_{I,f,r}(\bXi_{\quot{I}})} 
    =
    O_d(1)(2r)^{|I|}
    \cdot
     \log^{\frac{d}{2}}\left(\Pi_{i\in I} n\right)
     \cdot\max_{d_I \in \mathcal{T}_I}\norm{f_{I,d_I}}_\frob.
\end{equation}
holds. 

\medskip

We are now in position to conclude our proof of Theorem~\ref{thm:main-poly-2}. 
Substituting~\eqref{eq:bound-expectation-2}  into~\eqref{eq:small_ball_poly} and subsequently into the right hand side of~\eqref{eq:small_ball_est_bf}, we attain
{\small 
\begin{align*}
    &\sum_{I \subseteq [d]: I\neq\emptyset}
\sup_{x\in\R}\PROBp{ 
|f(\bXi) - x | 
\leq
\eps_{I,f,r}(\bXi_{\quot{I}})}
 =\\ 
 &
\sum_{I \subseteq [d]: I\neq\emptyset}
\frac{O_d((2r)^{|I|})
    \cdot
     \log^{\frac{d}{2}}\left(\Pi_{i\in I} n\right)
     \cdot\max_{d_I \in \mathcal{T}_I}\norm{f_{I,d_I}}_\frob}{ \|f\|_\frob}
  + \sum_{I \subseteq [d]: I\neq\emptyset}\min_{i \in I}\exp \left(-\Omega_d(1) \min_{\substack{k \in [2(d-1)]  \\ k\; \text{even}}} \left( \frac{ \|f\|^2_\frob}{\GameNorm}\right)^{1/\Theta_d(1)} \right).
\end{align*}
}
The assertion of Theorem~\ref{thm:main-poly-2} follows. 
\hfill$\blacksquare$

\subsubsection{Proof of Lemma~\ref{lem:poly-radius}}\label{sec:poly-raidus}

Our proof of this lemma has three ingredients; the first of which is seen in~\eqref{eq:Delta-initial-bound} below and is developed next.
We start by setting up the following notation. Given the sequence $(\bxi_i)_{i \in [d]}$ and two integers $1 \leq i, j \leq d$, define $\bXi[i:j] \coloneqq \bxi_i \circ \dots \circ \bxi_j$ , note that for $i,j \in [d]$ such that $j <i$ the notation $\bXi[i:j]$ indicates the empty sequence. Similarly, define $\bPsi[i:j] \coloneqq \bpsi_i \circ \cdots \circ \bpsi_j$ and $\bZeta[i:j] \coloneqq \bzeta_i \circ \cdots \circ \bzeta_j$, and for $I \subseteq [d]$ and $d_I \in \mathcal{T}_I$, define $\bdelta_{I,d_I}[i:j] \coloneqq \bdelta_{I,d_I}(i)\circ \cdots \circ \bdelta_{I,d_I}(j)$. Using this notation, write 
\begin{align}
    \Delta_{f,k}&(\ell_1,\ldots,\ell_d) \nonumber \\ 
    & = \max_{(\bpsi_i \in N_{\ell_i}(\bxi_i))_{i\in [d]}} \Big| f\big(\bPsi[1:d]) - f\big(\bXi[1:d]) \Big| \nonumber\\
    & = \max_{(\bpsi_i \in N_{\ell_i}(\bxi_i))_{i\in [d]}} \Big| \sum_{i=1}^d f\big(\bPsi[1:i] \circ \bXi[i+1:d] \big) - f\big(\bPsi[1:i-1] \circ \bXi[i:d]\big)  \Big| \nonumber \\
    & \leq \max_{(\bpsi_i \in N_{\ell_i}(\bxi_i))_{i\in [d]}} \sum_{i=1}^d \Big| f\big(\bPsi[1:i] \circ \bXi[i+1:d]\big) - f\big(\bPsi[1:i-1] \circ \bXi[i:d]\big) \Big| \nonumber \\
    & \leq \sum_{i=1}^d \max_{(\bpsi_i \in N_{\ell_i}(\bxi_i))_{i\in [d]}} \Big| f\big(\bPsi[1:i-1] \circ {\bm (}\bpsi_i - \bxi_i{\bm )} \circ \bXi[i+1:d]\big)\Big|,\label{eq:Delta-initial-bound}
\end{align}
where for the last inequality we utilise the convexity of maximisation as well as the multi-linearity of $f$. Inequality~\eqref{eq:Delta-initial-bound} essentially decomposes the total change affecting the value $f$ into a sum in which the flips are carried out one dimension after another. 

\medskip
The second ingredient of our proof is a bound on the change to the value of $f$ incurred through conducting flips along a single dimension. Given $i \in [d]$, the nonzero entries of $\bpsi_i-\bxi_i$ (the vector encountered in~\eqref{eq:Delta-initial-bound}) are the entries over which sign-flips are performed in $\bxi_i$; in that, $|(\bpsi_i-\bxi_i)_j| \in \{0,2\}$ holds for every $j \in [n_i]$. To estimate the effect of conducting $\ell_i$ sign-flips over $\bxi_i$, let 
$$
\bzeta_1,\ldots, \bzeta_d \in \{-1,0,1\}^{n_1}, \ldots, \bzeta_d \in \{-1,0,1\}^{n_d}
$$
be arbitrary\footnote{The need to allow for these vectors to have zero entries is in anticipation of future invocations of~\eqref{eq:single-dim-flip}.} and note that 
performing said sign-flips may alter the value of 
$$
f\big(\bZeta[1:i-1] \circ \bxi_i \circ \bZeta[i+1:d])
$$ 
by at most 
\begin{equation}\label{eq:single-dim-flip}
2\ell_i \max_{j \in [n_i]} \left| f\big( \bZeta[1:i-1] \circ \be_j \circ \bZeta[i+1:d] \big) \right|.
\end{equation}

The third ingredient of our proof is captured through the following claim. 

\begin{claim}\label{clm:switch}
Let $i \in [d]$ and $j \in [n_i]$ be fixed. Let $\big( \bpsi_k \in N_{\ell_k}(\bxi_k)\big)_{k \in [i-1]}$ and let $\big(\bzeta_{k} \in \{-1,0,1\}^{n_{k}}\big)_{k \in [i+1,d]}$ be arbitrary. 
Then,
\begin{align}
\Big|f \big(\bPsi[1:i-1]\circ \be_j \circ \bZeta[i+1:d] \big) \Big|
\leq \sum_{I \subseteq [i-1]} (2r)^{|I|} \max_{d_I \in \mathcal{T}_I} \left| f\big(\bdelta_{I,d_I}[1:i-1]\circ \be_j\circ \bZeta[i+1:d]  \big)\right|  \label{eq:switch} 
\end{align}
\end{claim}

Postponing the proof of Claim~\ref{clm:switch} until the end of this section, we proceed to deducing the assertion of the lemma from the aforementioned three ingredients. Indeed, equipped with these we may write 

\begin{align*}
\Delta_{f,k}(\ell_1,\ldots,\ell_d) & \leq \sum_{i=1}^d 2 \ell_i \max_{j \in [n_i]} \max_{\left(\bpsi_i \in N_{\ell_i}(\bxi_i) \right)_{i \in [d]}} \left|f \big(\bPsi[1:i-1 ]\circ \be_j \circ \bXi[i+1:d] \big) \right|\\
& \leq \sum_{i=1}^d 2 \ell_i \max_{j \in [n_i]} \sum_{I \subseteq [i-1]} (2r)^{|I|} \max_{d_I \in \mathcal{T}_I} \left| f\big(\bdelta_{I,d_I}[1:i-1]\circ \be_j\circ \bZeta[i+1:d]  \big)\right|\\
& \leq \sum_{i=1}^d (2r)^{i}\sum_{I \subseteq [i-1]} \max_{j \in [n_i]} \max_{d_I \in \mathcal{T}_I} \left| f\big(\bdelta_{I,d_I}[1:i-1]\circ \be_j\circ \bZeta[i+1:d] \big)\right| \\
& \leq \sum_{I \subseteq [d]} (2r)^{|I|} \max_{d_I \in \mathcal{T}_I} \left| f\big(\bdelta_{I,d_I}[1:d]\big) \right| \\
& = \sum_{I \subseteq [d]} (2r)^{|I|} \|f(\bXi)\|_{I,\infty},
\end{align*}
where the first inequality is owing to \eqref{eq:Delta-initial-bound} and~\eqref{eq:single-dim-flip}; the second inequality is supported by~\eqref{eq:switch}; for the third inequality we rely on $\ell_i \leq r$ and the fact that the inner sum in the preceding line ranges over subsets of $[i-1]$. For the penultimate inequality, note that the preceding sum ranges over all subsets of $[d]$ - for each $i \in [d]$ the sum ranges over all subsets of $[i-1]$; the maximisation over $[n_i]$ is absorbed by the maximisation over $d_I$.  The definition of $f(\bXi)_{I,\infty}$ delivers the last equality. 

\medskip
It remains to prove Claim~\ref{clm:switch}. 
\medskip

\noindent
{\em Proof of Claim~\ref{clm:switch}.} The proof is by induction on $i$ (i.e. the position of the standard base vector). For $i= 1$, the sum appearing on the right hand side of~\eqref{eq:switch} ranges only over $I = \emptyset$ so that~\eqref{eq:switch} trivially holds. Proceding to the induction step, assume that the claim holds for position $i-1$ and consider the claim for the $i$th position. Start by writing 
\begin{align*}
\Big| f \big(\bPsi[1:i-1 ]\circ \be_j \circ \bZeta[i+1:d] \big)\Big| & = \Big| f \big(\bPsi[1:i-1 ]\circ \be_j \circ \bZeta[i+1:d] \big)  - K + K \Big| \\
& \leq \Big|f \bPsi[1:i-1 ]\circ \be_j \circ \bZeta[i+1:d]) -K \Big| + |K|, 
\end{align*}
where 
$$
K := f \big(\bPsi[1:i-2]\circ\bxi_{i-1}\circ \be_j \circ \bZeta[i+1:d] \big). 
$$ 
Multilinearity of $f$ yields 
$$
f \big(\bPsi[1:i-1 ]\circ \be_j \circ \bZeta[i+1:d]) \big) - K = f \Big(\bPsi[1:i-2]\circ {\bm (}\bpsi_{i-1} - \bxi_{i-1}{\bm )}\circ \be_j \circ \bZeta[i+1:d] \Big)
$$
allowing us to write 
\begin{align}
\Big| f  \big(\bPsi[1:i-1 ]&\circ \be_j \circ \bZeta[i+1:d]) \big)\Big| \nonumber \\
 &\overset{\phantom{\eqref{eq:single-dim-flip}}}{\leq} \Big|f \Big(\bPsi[1:i-2]\circ {\bm (}\bpsi_{i-1} - \bxi_{i-1}{\bm )}\circ \be_j \circ \bZeta[i+1:d] \Big) \Big| + |K|. \nonumber\\
 & \overset{\eqref{eq:single-dim-flip}}{\leq}2 \ell_{i-1} \max_{k \in [n_{i-1}]} \Big| f\big(\bPsi[1:i-2]\circ\be_k \circ \be_j \circ \bZeta[i+1:d] \big)\Big| + |K|.\label{eq:induction split}
\end{align}
Applying the induction hypothesis on each of the two summands appearing on the right hand side of~\eqref{eq:induction split} yields  
\begin{align}
\label{eq:after-induction}
\Big| f  \big(\bPsi[1:i-1 ]\circ & \be_j \circ \bZeta[i+1:d]) \big)\Big| \leq \\
& 2 \ell_{i-1} \max_{k \in [n]} \sum_{I \subseteq [i-2]} (2r)^{|I|} \max_{d_I \in \mathcal{T}_I} \Big|f \big( \bdelta_{I,d_I}[1:i-2] \circ \be_k \circ \be_j \circ \bZeta[i+1:d] \big) \Big| \nonumber \\
& +  \nonumber \\ 
&\sum_{I \subseteq [i-2]} (2r)^{|I|} \max_{d_I \in \mathcal{T}_I} \Big|f \big(\bdelta_{I,d_I}[1:i-2] \circ \bxi_{i-1} \circ \be_j \circ \bZeta[i+1:d] \big) \Big|. \nonumber
\end{align}
To conclude, note that both sums appearing on the right hand side of~\eqref{eq:after-induction} together do not exceed
\begin{equation}\label{eq:induction-end}
\sum_{I \subseteq [i-1]} (2r)^{|I|} \max_{d_I \in \mathcal{T}_I} \Big|f \big(\bdelta_{I,d_I}[1:i-2]\circ \bdelta_{I,d_I}(i-1) \circ \be_j \circ  \bZeta[i+1:d] \big) \Big|.
\end{equation}
To see this, note that the first of these sums can be viewed as ranging over all subsets of $[i-1]$ containing the element $i-1$ and thus the replacement of $\bm e_k$ by $\bdelta_{I,d_I}(i-1)$ leads to 
\begin{align}
2 \ell_{i-1} & \max_{k \in [n]} \sum_{I \subseteq [i-2]} (2r)^{|I|} \max_{d_I \in \mathcal{T}_I} \Big|f \big( \bdelta_{I,d_I}[1:i-2] \circ \be_k\circ \be_j \circ \bZeta[i+1:d] \big) \Big| \label{eq:first-sum} \\
& \overset{\phantom{\ell_{i-1} \leq k}}{\leq} 2 \ell_{i-1}  \sum_{\substack{I \subseteq [i-1]:\\ i-1 \in I}} (2r)^{|I|-1} \max_{d_I \in \mathcal{T}_I} \Big|f \big( \bdelta_{I,d_I}[1:i-2] \circ \bdelta_{I,d_I}(i-1)\circ \be_j \circ \bZeta[i+1:d]  \big) \Big| \nonumber \\
& \overset{\ell_{i-1} \leq r}{\leq} \sum_{\substack{I \subseteq [i-1]:\\ i-1 \in I}} (2r)^{|I|} \max_{d_I \in \mathcal{T}_I} \Big|f \big( \bdelta_{I,d_I}[1:i-2] \circ \bdelta_{I,d_I}(i-1)\circ \be_j \circ \bZeta[i+1:d]  \big) \Big|, \nonumber 
\end{align}
where here the maximisation over $[n_i]$ is accounted for through the maximisaiton over $d_I$ which now ranges over tuples of size $i-1$ with the dimension $i-1$ included. 

The second sum can be viewed as a sum over the subsets of $[i-1]$ not containing the element $i-1$ allowing for the replacement of $\bxi_{i-1}$ with $\bdelta_{I,d_I}(i-1)$ and thus yielding 
\begin{align}
 \sum_{I \subseteq [i-2]} (2r)^{|I|} & \max_{d_I \in \mathcal{T}_I} \Big|f \big(\bdelta_{I,d_I}[1:i-2] \circ \bxi_{i-1} \circ \be_j \circ  \bZeta[i+1:d] \big) \Big| \label{eq:2nd-sum}\\
 & = \sum_{\substack{I \subseteq [i-1]:\\ i-1 \notin I}} (2r)^{|I|}  \max_{d_I \in \mathcal{T}_I} \Big|f \big(\bdelta_{I,d_I}[1:i-2] \circ \bdelta_{I,d_I}(i-1) \circ \be_j \circ  \bZeta[i+1:d]\big) \Big|.\nonumber
\end{align}
The combination of~\eqref{eq:first-sum} and~\eqref{eq:2nd-sum} yields~\eqref{eq:induction-end} and concludes the proof of Claim~\ref{clm:switch}.\hfill$\square$
\bigskip

\noindent
This concludes our proof of Lemma~\ref{lem:poly-radius}. \hfill $\blacksquare$

\subsubsection{Proof of Lemma~\ref{lem:small-ball-multilinear}}\label{sec:lem:small-ball-multilinear}

Let a nonempty $I \subseteq [d]$ be given and fix an arbitrary $i \in I$. Apply the Law of Total Probability (twice) as to write
{\small
\begin{align}   
\sup_{x\in\R}&\PR \Big\{|f\big(\bXi \big)- x | \leq 
\eps_{I,f,r}(\bXi_{\quot{I}}) \Big\} 
\leq \nonumber  \\
& \sup_{x\in\R} \PROBp{|f\big(\bXi \big)- x | \leq 
\eps_{I,f,r}(\bXi_{\quot{I}}) \; \Big\vert\; \norm{\bv_{f,i}(\bXi_{\quot{i}})}_2 \geq \frac{\norm{f}_\frob}{2}} + \PR_{\bXi_{\quot{i}}} \left\{ \norm{\bv_{f,i}(\bXi_{\quot{i}})}_2 < \frac{\norm{f}_\frob}{2}\right\} \leq \nonumber\\
& \EXP_{\bXi_{\quot{i}}}\left[\sup_{x\in\R}\PR_{\bxi_i}\left\{
|f\big(\bXi\big)- x | \leq 
\eps_{I,f,r}(\bXi_{\quot{I}}) 
\,\middle\vert\, \norm{\bv_{f,i}(\bXi_{\quot{i}})}_2 \geq  \frac{\norm{f}_\frob}{2} \;,\; 
\bXi_{\quot{i}}
\right\}\right] + \PR_{\bXi_{\quot{i}}}\left\{
\norm{\bv_{f,i}(\bXi_{\quot{i}})}_2 \leq  \frac{\norm{f}_\frob}{2}
\right\},\label{eq:conditioning}
\end{align}
}
where we recall that $\bXi_{\quot{i}}$ denotes the sequence of Rademacher vectors obtained from $\bXi$ by omitting $\bxi_i$. 

Owing to Lemma~\ref{lem:v-concentration}, 
\begin{equation}\label{eq:concentrated-frob-poly}
\PROBp{
\norm{\bv_{f,i}(\bXi_{\quot{i}})}_2 \leq  \frac{\norm{f}_\frob}{2}} \leq \exp \left(-\Omega_d(1) \min_{\substack{k \in [2(d-1)]  \\ k\; \text{even}}} \left( \frac{ \|f\|^2_\frob}{\GameNorm}\right)^{1/\Theta_d(1)} \right).
\end{equation}

Proceeding to the term appearing within the expectation seen on the right hand side of~\eqref{eq:conditioning}, recall that the vector
$\bv_{f,i}(\bXi_{\quot{i}})$ depends not on the vector $\bxi_i$ but only on the members of $\bXi_{\quot{i}}$.
Fix then a realisation $\Xi_{\quot{i}}$ of the Rademacher tensor $\bXi_{\quot{i}}$ for which the now fully determined vector $\bv_{f,i}(\Xi_{\quot{i}}) $ satisfies $\|\bv_{f,i}(\Xi_{\quot{i}})\|_2 \geq \|f\|_\frob/2 >0 $; the last inequality is owing to $\|f\|_\frob >0$ assumed in the premise of Theorem~\ref{thm:main-poly} through $f \not\equiv 0$. With $\bXi_{\quot{i}}$ fixed (to be $\Xi_{\quot{i}}$), the chaos $f(\bXi)$  reduces to a sum of $n_i$ independent random variables 
denoted $\tilde f(\bxi_i)$ and given by 
$$
\tilde f(\bxi_i) := \sum_{j \in [n_i]}(\bxi_i)_j \cdot \bv_{f,i}(\Xi_{\quot{i}})_j,
$$
where we recall that $\bv_{f,i}(\Xi_{\quot{i}})_j =  f_{\{i\},\{j\}}(\Xi_{\quot{i}})$. The random variable $\eps_{I,f,r}(\bXi_{\quot{I}})$ 
depends solely on the members of $\bXi_{\quot{I}}$. Hence,  the fixation $\bXi_{\quot{i}} = \Xi_{\quot{i}}$ completely determines this random variable and we write $\tilde \eps_{I,f,r}(\Xi_{\quot{I}})$ to denote its value associated with the realisation $\Xi_{\quot{i}}$. Interest then shifts towards obtaining an upper bound on 
$$
\sup_{x\in\R}\PR_{\bxi_i}\left\{
|\tilde f\big(\bxi_i\big)- x | \leq 
\tilde \eps_{I,f,k}(\Xi_{\quot{I}}) 
\,\middle\vert\,  \bXi_{\quot{i}} = \Xi_{\quot{i}}, \;\norm{\bv_{f,i}(\Xi_{\quot{i}})}_2 \geq  \frac{\norm{f}_\frob}{2}
\right\}
$$
this we obtain through an application of the Kolmogorov-Rogozin inequality, namely Theorem~\ref{thm:KRI}.

Gearing up towards such an application, set
$$
X_j := (\bxi_i)_j \cdot \bv_{f,i}(\Xi_{\quot{i}})_j, \; \eps:= \tilde \eps_{I,f,r}(\Xi_{\quot{I}}), \; \text{and} \; \eps_j := |\bv_{f,i}(\Xi_{\quot{i}})_j|/2
$$
for each $j \in [n]$, where $X_j$, $\eps$, and $\eps_j$ are per Theorem~\ref{thm:KRI}. Owing to $\|\bv_{f,i}(\Xi_{\quot{i}})\|_2 > 0$, there exists a $j \in [n]$ for which $\eps_j$ is non-zero. For each such $j$, we may write that $\levy_{\eps_j}(X_j) =1/2$ as the mass of $X_j$ is supported on $\pm |\bv_{f,i}(\Xi_{\quot{i}})_j|$. All this collectively yields 
$$
\sum_{j=1}^{n} \eps_j^2 \big(1-\levy_{\eps_j}(X_j)\big) = \frac{1}{8} \sum_{j \in [n_i]} |\bv_{f,i}(\Xi_{\quot{i}})_j|^2 = \|\bv_{f,i}(\Xi_{\quot{i}})\|_2^2 / 8. 
$$

By definition, $\eps_j  < \|f_{i}(\Xi_{\quot{i}})\|_{,\infty}$ holds for every $j \in [n]$. Then, owing to the definition of $\eps_{I,f,r}(\bXi_{\quot{I}})$ seen in~\eqref{eq:radius-poly-I} as well as the fact that $i \in I$, it follows that $\tilde \eps_{I,f,r}(\Xi_{\quot{I}}) > \eps_j$ holds for every $j \in [n]$ whenever $r \geq 1$. We may thus apply Theorem~\ref{thm:KRI} which in turn yields
$$
\sup_{x\in\R}\PR_{\bxi_i} \left\{
|f\big(\bXi\big)- x | \leq 
\tilde \eps_{I,f,r}(\Xi_{\quot{I}}) 
\,\middle\vert\,  
\bXi_{\quot{i}} = \Xi_{\quot{i}} \;,\; \norm{\bv_{f,i}(\Xi_{\quot{i}})}_2 \geq  \frac{\norm{f}_\frob}{2} 
\right\} 
 \leq \frac{C' \cdot \tilde \eps_{I,f,r,}(\Xi_{\quot{I}})}{\|\bv_{f,i}(\Xi_{\quot{i}})\|_2}  \leq  \frac{C \cdot \tilde \eps_{I,f,r}(\Xi_{\quot{I}})}{\|f\|_\frob}, 
$$
where $C,C' >0$ are constants arising from  Theorem~\ref{thm:KRI} (and consequently independent of $d$). 

To conclude the proof of this lemma, substitute~\eqref{eq:concentrated-frob-poly} as well as the last obtained bound into the right hand side of~\eqref{eq:conditioning} as to reach 
\begin{align*}
\sup_{x\in\R}\PR \Big\{&|f\big(\bXi \big)- x | \leq 
\eps_{I,f,r}(\bXi_{\quot{I}}) \Big\} \leq \\ 
& \frac{C\Ex_{\bXi_{\quot{i}}} \left\{\eps_{I,f,r}(\bXi_{\quot{I}}) \right\}}{\|f\|_\frob} + \exp \left(-\Omega_d(1) \min_{\substack{k \in [2(d-1)]  \\ k\; \text{even}}} \left( \frac{ \|f\|^2_\frob}{\GameNorm}\right)^{1/\Theta_d(1)} \right)   
\end{align*}
concluding the proof of the lemma. \hfill$\blacksquare$

\subsubsection{Proof of Lemma~\ref{lem:bound-expectation}}\label{sec:lem:bound-expectation}
Let $I \subseteq [d]$ be given. Start with 
$$
\Ex \left\{\eps_{I,f,r}(\bXi_{\quot{I}})\right\} \overset{\eqref{eq:radius-poly-I}}{=} 2^d(2r)^{|I|} \Ex \left\{\|f_{I}(\bXi_{\quot{I}})\|_{\infty}\right\} \overset{\eqref{eq:res-inf-norm}}{=} 2^d(2r)^{|I|} \Ex \left\{\max_{d_I \in \mathcal{T}_I} |f_{I,d_I}(\bXi_{\quot{I}})| \right\}.
$$
Then, recalling $\Phi$ from~\eqref{eq:Phi}, note that the inequality
\begin{align}
    \Ex \left\{ \max_{d_I \in \mathcal{T}_I} |f_{I,d_I}(\bXi_{\quot{I}})| \right\} &= \Ex \left\{\Phi^{-1}\left(\Phi\left( \max_{d_I \in \mathcal{T}_I} \frac{\lambda |f_{I,d_I}(\bXi_{\quot{I}})|}{\lambda}\right)\right)\right\} \nonumber \\
    & \leq \lambda \Phi^{-1} \left( \Ex \left\{ \Phi \left( \max_{d_I \in \mathcal{T}_I} \frac{|f_{I,d_I}(\bXi_{\quot{I}})|}{\lambda} \right)\right\}\right) \nonumber \\
    & \leq \lambda \Phi^{-1} \left(\Ex \left\{ \sum_{d_I \in \mathcal{T}_I} \Phi \left(\frac{|f_{I,d_I}(\bXi_{\quot{I}})|}{\lambda}\right)\right\} \right) \nonumber \\
    & = \lambda \Phi^{-1} \left( \sum_{d_I \in \mathcal{T}_I} \Ex \left\{ \Phi \left(\frac{|f_{I,d_I}(\bXi_{\quot{I}})|}{\lambda}\right)\right\} \right) \label{eq:res-chaos-exp-bound}
\end{align}
holds for any $\lambda >0$, where the second inequality is owing to Jensen's inequality and $\Phi^{-1}$ being concave. Owing to~\cite[Equation~(4.3.4)]{de2012decoupling} as well as our choice for $\alpha$ in defining $\Phi$ (see~\eqref{eq:Phi}), there exists a quantity $C_d := C_d(d)$ for which the inequality 
$$
\inf \left\{c > 0: \Ex \left\{ \Phi \left(\frac{|f_{I,d_I}(\bXi_{\quot{I}})|}{c}\right) \right\} \leq 1 \right\} \leq C_d \cdot \|f_{I,d_I}\|_\frob
$$
holds. Returning to~\eqref{eq:res-chaos-exp-bound} with $\lambda = C_d \cdot \max_{d_I \in \mathcal{T}_I}\|f_{I,d_I}\|_\frob$ yields 
$$
\Ex \left\{ \max_{d_I \in \mathcal{T}_I} |f_{I,d_I}(\bXi_{\quot{I}})| \right\} \leq C_d \cdot \max_{d_I \in \mathcal{T}_I}\|f_{I,d_I}\|_\frob \cdot \Phi^{-1} \left( |\mathcal{T}_d|\right) = O_d(1) \cdot \max_{d_I \in \mathcal{T}_I}\|f_{I,d_I}\|_\frob \cdot \Phi^{-1} \left( \Pi_{i \in I}n\right)
$$
concluding the proof of the lemma. \hfill$\blacksquare$
\bibliographystyle{plain}
\bibliography{ref}

\begin{thebibliography}{10}

\bibitem{AW15}
R.~Adamczak and P.~Wolff.
\newblock Concentration inequalities for non-{L}ipschitz functions with bounded
  derivatives of higher order.
\newblock {\em {Probability Theory and Related Fields}}, 162(3-4):531--586,
  2015.

\bibitem{BKW22}
S.~Bamberger, F.~Krahmer, and R.~Ward.
\newblock The {H}anson-{W}right inequality for random tensors.
\newblock {\em Sampling Theory, Signal Processing, and Data Analysis},
  20(2):Paper No. 14, 35, 2022.

\bibitem{BFK17}
A.~Bandeira, A.~Ferber, and M.~Kwan.
\newblock {Resilience for the Littlewood-Offord problem}.
\newblock {\em {Advances in Mathematics}}, 319:292--312, 2017.

\bibitem{BVWM10}
J.~Bourgain, V.~Vu, and P.~Wood.
\newblock On the singularity probability of discrete random matrices.
\newblock {\em {Journal of Functional Analysis}}, 258:559--603, 2010.

\bibitem{CTV06}
K.~Costello, T.~Tao, and V.~Vu.
\newblock Random symmetric matrices are almost surely nonsingular.
\newblock {\em {Duke Mathematical Journal}}, 135:395--413, 2006.

\bibitem{de2012decoupling}
V.~de~la Pe{\~n}a and E.~Gin{\'e}.
\newblock {\em {Decoupling: From Dependence to Independence}}.
\newblock {Probability and Its Applications}. Springer New York, 2012.

\bibitem{Erdos45}
P.~Erd\H{o}s.
\newblock {On a lemma of Littlewood and Offord}.
\newblock {\em {Bulletin of the American Mathematical Society}}, 51:898--902,
  1945.

\bibitem{Erdos65}
P.~Erd\H{o}s.
\newblock Extremal problems in number theory.
\newblock In {\em Proceedings of the symposium for pure mathematics}, volume
  VIII, pages 181--189. {American Mathematical. Society, Providence, RI}, 1965.

\bibitem{FLM21}
A.~Ferber, K.~Luh, and G.~McKinley.
\newblock Resilience of the rank of random matrices.
\newblock {\em Combinatorics, Probability and Computing}, 30:163--174, 2021.

\bibitem{Hal77}
G.~Hal\'{a}sz.
\newblock Estimates for the concentration function of combinatorial number
  theory and probability.
\newblock {\em {Periodica Mathematica Hungarica. Journal of the J\'{a}nos
  Bolyai Mathematical Society}}, 8:197--211, 1977.

\bibitem{KKS95}
J.~Kahn, J.~Koml\'os, and E.~Szemer\'edi.
\newblock On the probability that a random $\pm 1$-matrix is singular.
\newblock {\em {Journal of the American Mathematical Society}}, 8:223--240,
  1995.

\bibitem{KR19}
S.~P. Kasiviswanathan and M.Rudelson.
\newblock Restricted isometry property under high correlations, 2019.
\newblock Arxiv preprint arXiv:1904.05510.

\bibitem{KRu18}
S.~P. Kasiviswanathan and M.~Rudelson.
\newblock {Restricted Eigenvalue from Stable Rank with Applications to Sparse
  Linear Regression}.
\newblock In {\em {Proceedings of the 31st Conference On Learning Theory,
  PMLR}}, volume~75, pages 1011--1041, 2018.

\bibitem{KV04}
J.~Kim and V.~Vu.
\newblock Sandwiching random graphs: universality between random graph models.
\newblock {\em {Advances in Mathematics}}, 188:444--469, 2004.

\bibitem{KSV02}
J.~H. Kim, B.~Sudakov, and V.~H. Vu.
\newblock On the asymmetry of random regular graphs and random graphs.
\newblock volume~21, pages 216--224. 2002.

\bibitem{kolmogorov1958proprietes}
A.~Kolmogorov.
\newblock Sur les propri{\'e}t{\'e}s des fonctions de concentrations de mp
  l{\'e}vy.
\newblock In {\em {Annales de l'institut Henri Poincar{\'e}}}, volume~16, pages
  27--34, 1958.

\bibitem{Komlos67}
J.~Koml\'{o}s.
\newblock On the determinant of $(0,1)$ matrices.
\newblock {\em {Studia Scientiarum Mathematicarum Hungarica}}, 2:7--21, 1967.

\bibitem{LS12}
C.~Lee and B.~Sudakov.
\newblock Dirac's theorem for random graphs.
\newblock {\em {Random Structures \& Algorithms}}, 41(3):293--305, 2012.

\bibitem{LO38}
J.~Littlewood and A.~Offord.
\newblock On the number of real roots of a random algebraic equation.
\newblock {\em {The Journal of the London Mathematical Society}}, 13:288--295,
  1938.

\bibitem{Lovett}
S.~Lovett.
\newblock {An elementary proof of anti-concentration of polynomials in Gaussian
  variables}, 2010.
\newblock {Electronic Colloquium on Computational Complexity, Report No. 182 }.

\bibitem{M19}
R.~Montgomery.
\newblock Hamiltonicity in random graphs is born resilient.
\newblock {\em {Journal of Combinatorial Theory. Series B}}, 139:316--341,
  2019.

\bibitem{rogozin1961estimate}
B.~A. Rogozin.
\newblock An estimate for concentration functions.
\newblock {\em {Theory of Probability \& Its Applications}}, 6(1):94--97, 1961.

\bibitem{rogozin1961increase}
B.~A. Rogozin.
\newblock On the increase of dispersion of sums of independent random
  variables.
\newblock {\em {Theory of Probability \& Its Applications}}, 6(1):97--99, 1961.

\bibitem{vershynin2020concentrationinequalitiesrandomtensors}
V.~Roman.
\newblock {Concentration inequalities for random tensors}.
\newblock {\em Bernoulli}, 26(4):3139 -- 3162, 2020.

\bibitem{RV08}
M.~Rudelson and R.~Vershynin.
\newblock {The Littlewood-Offord problem and invertibility of random matrices}.
\newblock {\em {Advances in Mathematics}}, 218:600--633, 2008.

\bibitem{rudelson2013hanson}
M.~Rudelson and R.~Vershynin.
\newblock Hanson-{W}right inequality and sub-gaussian concentration, 2013.

\bibitem{SS65}
A.~S\'ark\"ozi and E.~Szemer\'edi.
\newblock {\"Uber ein problem von Erd\H{o}s und Moser}.
\newblock {\em {Acta Arithmetica}}, 11:205--208, 1965.

\bibitem{Su16}
B.~Sudakov.
\newblock Robustness of graph properties.
\newblock In {\em Surveys in combinatorics 2017}, volume 440 of {\em {London
  Math. Soc. Lecture Note Ser.}}, pages 372--408. {Cambridge Univ. Press,
  Cambridge}, 2017.

\bibitem{SV08}
B.~Sudakov and V.~Vu.
\newblock Local resilience of graphs.
\newblock {\em {Random Structures \& Algorithms}}, 33:409--433, 2008.

\bibitem{T12}
T.~Tao.
\newblock {\em Topics in random matrix theory}, volume 132 of {\em Graduate
  Studies in Mathematics}.
\newblock American Mathematical Society, Providence, RI, 2012.

\bibitem{TV06}
T.~Tao and V.~Vu.
\newblock On random $\pm 1$-matrices: singularity and determinant.
\newblock {\em {Random Structures \& Algorithms}}, 28:1--23, 2006.

\bibitem{TV07}
T.~Tao and V.~Vu.
\newblock {On the singularity probability of random Bernoulli matrices}.
\newblock {\em {Journal of the American Mathematical Society}}, 20:603--628,
  2007.

\bibitem{TV09}
T.~Tao and V.~Vu.
\newblock {Inverse Littlewood-Offord theorems and the condition number of
  random discrete matrices}.
\newblock {\em {Annals of Mathematics}}, 169:595--632, 2009.

\bibitem{TaoVuBook}
T.~Tao and V.~Vu.
\newblock {\em {Additive Combinatorics}}, volume 105 of {\em {Cambridge Studies
  in Advanced Mathematics}}.
\newblock {Cambridge University Press, Cambridge}, 2010.

\bibitem{Tik}
K.~Tikhomirov.
\newblock {Singularity of random Bernoulli matrices}.
\newblock {\em {Annals of Mathematics. Second Series}}, 191:593--634, 2020.

\bibitem{van2014probability}
R.~van Handel.
\newblock Probability in high dimensions.
\newblock {\em {Lecture Notes (Princeton University)}}, 2014.

\bibitem{Vershynin}
R.~Vershynin.
\newblock {\em High-dimensional probability}, volume~47 of {\em {Cambridge
  Series in Statistical and Probabilistic Mathematics}}.
\newblock {Cambridge University Press, Cambridge}, 2018.

\bibitem{Vu08}
V.~Vu.
\newblock Random discrete matrices.
\newblock In {\em Horizons of combinatorics, Bolyai Soc. Math. Stud.},
  volume~17, pages 257--280. Springer, 2008.

\bibitem{Vu21}
V.~Vu.
\newblock Recent progress in combinatorial random matrix theory.
\newblock {\em Probability Surveys}, 18:179--200, 2021.

\end{thebibliography}

\appendix 

\section{Deducing Corollary~\ref{cor:transparent} from Theorem~\ref{thm::main-bf}}\label{par:deduce} Let $M$ be as in the premise of Corollary~\ref{cor:transparent}. In particular, the assumption that $\|M\|_\infty =1$ implies that $\|M\|_{\infty,2} \leq \sqrt{\|M\|_{\infty,0}}$, leading to 
\begin{align} 
f(M,n) \leq \frac{\min \{\|M\|_{\infty,0}, \sqrt{\|M\|_{\infty,0}\log n} \}}{\|M\|_\frob},
\qquad g(M,r) = \frac{\min\{r,\|M\|_{\infty,0}\}}{\|M\|_\frob}.\label{eq:f-and-g}
\end{align}
The term $\min \left\{\|M\|_{\infty,0}, \sqrt{\|M\|_{\infty,0}\log n}\right\}$ (seen in the nominator of $f(M,n)$ in~\eqref{eq:f-and-g}) compels us to distinguish  between two regimes, namely a {\em sparse} regime and a {\em dense} one, as defined in Corollary~\ref{cor:transparent}. Analysis of resilience guarantees in each such regime is as follows.   

\bigskip
\noindent 
{\em Sparse regime.} Given a matrix $M \not\equiv 0$ in this regime satisfying $\sr(M) = \omega(1)$, we seek to determine the largest $\ress$ for which $\sup_{x \in \R} \PROBp{\res_x^M(\bpsi,\bxi) \leq r} = o(1)$. By Theorem~\ref{thm::main-bf} it suffices to require that 
$\ress \cdot f(M,n) = o(1)$ as well as $r\cdot g(M,r) = o(1)$. The restriction 
\begin{equation}\label{eq:k-first-cond}
\frac{r\cdot \|M\|_{\infty,0}}{\|M\|_\frob}  = o(1)
\end{equation}
is imposed by $r \cdot f(M,n) = o(1)$. Subject to $\ress$ satisfying~\eqref{eq:k-first-cond}, the equality $r\cdot g(M,r) = o(1)$ asserts that we seek $\ress$ for which 
$$
\min\left\{\frac{r \|M\|_{\infty,0}}{\|M\|_\frob} \;,\;  \frac{r^2}{\|M\|_\frob} \right\} = o(1)
$$
holds. Overall we reach that any $M$ in the sparse regime a.a.s.\ has resilience as high as 
$$
o\left(\min \left\{\frac{\|M\|_\frob}{\|M\|_{\infty,0}}\;,\;\max \left\{\frac{\|M\|_\frob}{\|M\|_{\infty,0}}, \sqrt{\|M\|_\frob} \right\} \right\}\right) = o\left(\frac{\|M\|_\frob}{\|M\|_{\infty,0}} \right);
$$
recovering the probabilistic resilience guarantee asserted in Corollary~\ref{cor:transparent} for the sparse regime.

\bigskip
\noindent 
{\em Dense regime.} Similar analysis to the one performed in the sparse regime reveals that a matrix $M$ in the dense regime satisfying $\sr(M) = \omega(1)$ a.a.s.\ has resilience given by 
$$
o \left( \min \left\{\frac{\|M\|_\frob}{\sqrt{\|M\|_{\infty,0}\log n}} \;,\; \max \left\{\frac{\|M\|_\frob}{\|M\|_{\infty,0}}, \sqrt{\|M\|_\frob}  \right\}\right\}\right).
$$
In this regime, $\frac{\|M\|_\frob}{\|M\|_{\infty,0}} \leq \frac{\|M\|_\frob}{\sqrt{\|M\|_{\infty,0}\log n}}$. If the former prevails in the maximisation of the last display, i.e. if $\|M\|_\frob > \|M\|_{\infty,0}^2$ holds, then the same term previals in the minimisation yielding Option~2(a) seen in Corollary~\ref{cor:transparent}; otherwise, Option~2(b) is reached. This concludes our proof of Corollary~\ref{cor:transparent}.  

\section{Resilience of high-degree block-diagonal tensors}\label{par:block-diagonal} In this section, we prove Claim~\ref{clm:block-diagonal}. 
The argument proposed has two distinct parts. The first handles the exponential seen on the right hand side of~\eqref{eq:main-poly-res} and establishes that the latter vanishes for $\fbl$ rendering Theorem~\ref{thm:main-poly} meaningful for $\fbl$. The second part deals with the resilience estimation through the asymptotic magnitude of the sum appearing on the right hand side of~\eqref{eq:main-poly-res}.  

Starting with the exponential seen on the right hand side of~\eqref{eq:main-poly-res}, we prove that under the assumptions seen in the premise of the claim,
\begin{equation}\label{eq:exp-vanish}
\frac{\|\fbl\|_\frob^2}{\lVert \bm\Game \fbl \rVert} = \min_{m \in 2[d-1] } \frac{ \|\fbl\|^2_\frob}{\lVert\bm{\Game}_{i,m}\fbl\rVert} = \omega(1) 
\end{equation}
holds for every $i \in [d]$. This, in turn, yields that the aforementioned exponential vanishes in~\eqref{eq:main-poly-res}. To establish~\eqref{eq:exp-vanish}, it is more conducive to handle it under the fixation $i = d$ thus leading to simpler notation; the same argument holds for $i \in [d-1]$ as well by symmetry. 

The symmetry of $\fbl$ equips us with a useful property of $\mathcal{A} : = \mathcal{A}_d^{\fbl}$ which we record next. 
For any $k\in[n]$ and $\bj\in\indset{d-1}$,
\begin{align*}
\mathcal{A}_{k,\bj} =
\begin{cases}
    \ell, & (a-1)\blw+1\leq k,\bj \leq a \blw \\
    0, & \textup{otherwise},
\end{cases}
\end{align*}
where $a \in [n/\blw]$;
here we introduced the scaling factor $\ell$ in order to demonstrate that it plays no role in the resilience estimation performed using Theorem~\ref{thm:main-poly} which have been stated with $\ell=1$. 
So for any $\bj,\bk\in\indset{d-1}$,
\begin{equation}\label{eq:A3}
\left(\mathcal{A}^\trans \mathcal{A} \right)_{\bj\bk} = 
\begin{cases}
    \blw\ell^2, & (a-1)\blw+1\leq \bj,\bk \leq a \blw \\
    0, & \text{otherwise}.
\end{cases}
\end{equation}
 In that, the entry of $\mathcal{A}^\trans \mathcal{A}$ specified by the indices $\bj$ and $\bk$ vanishes unless the corresponding fibres of $\fbl$ specified by these indices lie in the same block of the tensor. 

With this understanding, we proceed to estimate the quantities $\lVert\bm{\Game}_{d,m}\fbl\rVert$ for $m\in2[d-1]$.
Equation~\eqref{eq:partial} reads 
\begin{equation}\label{eq:partial-2nd}
\lVert\bm{\Game}_{d,m}\fbl\rVert^2 = \sum_{\bell\in \Rel_{d,m}}
\left(
\sum_{(\bj,\bk) \in \Ext_{d,m}(\bell)} 
(\mathcal{A}^\trans\mathcal{A}              )_{\bj \bk}
\right)^2. 
\end{equation}
By \eqref{eq:A3}, for any 
$\bell=\big((b_1,c_1),\dots,(b_{\frac m 2},c_{\frac m 2}),(b_1,c'_1),\dots,(b_{\frac m 2},c'_{\frac m 2})\big)\in \Rel_{d,m}$,
we may restrict the internal sum to pairs $(\bj,\bk)\in\Ext_{d,m}(\bell)$ whose coordinates are all in the same block, that is, $(a-1)\blw+1\leq \bj,\bk \leq a \blw$ for some $a \in [n/\blw]$.
In addition, since $(\bj,\bk)$ is an extension of $\bell$, the latter must correspond to this same block $a$, namely,
$(a-1)\blw+1\leq c_1,\dots,c_{\frac m 2},c'_1,\dots,c'_{\frac m 2} \leq a \blw$.
As there are $O_d(\blw^{d-1-\frac m 2})$ such extensions,
$$
\sum_{(\bj,\bk) \in \Ext_{d,m}(\bell)} 
(\mathcal{A}^\trans\mathcal{A}              )_{\bj \bk}
= \blw\ell^2 \cdot
O_d(\blw^{d-1-\frac m 2})
= O_d(\ell^2\blw^{d-\frac m 2}).
$$
Since there are $n/\blw$ blocks and each block has $O_d(\blw^m)$ different $\bell$s,
the sum is
\begin{align*}
\sum_{\bell\in \Rel_{d,m}}
\left(
\sum_{(\bj,\bk) \in \Ext_{d,m}(\bell)} 
(\mathcal{A}^\trans\mathcal{A}              )_{\bj \bk}
\right)^2
= \frac n \blw\cdot O_d(\blw^m) \cdot 
\big(O_d(\ell^2\blw^{d-\frac m 2})\big)^2
=O_d(\ell^4 n\blw^{2d-1}).
\end{align*}

\medskip
Returning to~\eqref{eq:exp-vanish}, write $\|\fbl\|_\frob^2 = \frac{n}{\blw}\blw^d \ell^2 = n \blw^{d-1}  \ell^2$ and note that owing to the stipulation $\blw = o(n)$ appearing in the premise,  
\begin{align*}
\frac{\|\fbl\|_\frob^2}{\lVert\bm{\Game}_{d,m}\fbl\rVert} & 
= 
\Omega\left(\frac{\ell^2 n \blw^{d-1}  }{\sqrt{\ell^4 n\blw^{2d-1}}}\right) = \Omega \left(\sqrt{\frac{n}{\blw}}\right) = \omega(1)  .
\end{align*}

\medskip
Having established that the exponential appearing on the right hand side of~\eqref{eq:main-poly-res} has $o(1)$ order of magnitude, we turn to the sum appearing on the right hand side of~\eqref{eq:main-poly-res} and through which attain the proclaimed estimates for the resilience of $\fbl(\bxi_1,\dots,\bxi_d)$. Noting that for $\emptyset\neq I \subseteq [d]$ we have 
$\max_{d_I \in 
\indset{|I|}
}
\norm{\fbl_{I,d_I}}_\frob=\sqrt{\ell^2\blw^{d-|I|}}$,
\begin{align*}
\sum_{I \subseteq [d]: I\neq\emptyset}
r^{|I|}\cdot \max_{d_I \in 
\indset{|I|}
}
\frac{\norm{\fbl_{I,d_I}}_\frob}{\norm{\fbl}_\frob}
&=
\sum_{I \subseteq [d]: I\neq\emptyset}
r^{|I|}\cdot
\frac{\sqrt{\ell^2\blw^{d-|I|}}}{\sqrt{\ell^2 n \blw^{d-1}}}
\\&=
\sqrt{\frac{\blw}{n}}\sum_{I \subseteq [d]: I\neq\emptyset}
\left(\frac{r}{\sqrt{\blw}}\right)^{|I|}
=
\sqrt{\frac{\blw}{n}}
\left(\left(1+\frac{r}{\sqrt{\blw}}\right)^d
-1\right).
\end{align*}
This concludes our proof of Claim~\ref{clm:block-diagonal}. 

\section{Resilience of quadratic Rademacher chaos}\label{sec:quadratic-app}

The aim of this section is to make good on the claims made in Remark~\ref{rem:quadratic}. The main result of this section, reads as follows. 

\begin{theorem}
\label{thm:qf}
Let $0 \not\equiv M \in \mathbb{R}^{n \times n}$. Then, there exist constants $c_1,c_2,c_3,c_4>0$ such that for any integer $0< \ress \leq n$,
\begin{equation}
\label{eq:bound_qf}
\sup_{x \in \R} \PROB\left\{ \res_x^M(\bxi) \leq \ress\right\} \leq \Big(c_1 \ress \cdot f(M,n) + c_2\ress\cdot g(M,\ress) + \exp{\big(-c_3 \sr(M) \big)}\Big)^{1/4}
+ \frac{c_4}{n},
\end{equation}
where $\bxi \in \{\pm 1\}^n$ is a Rademacher vector and the chaos has the form $\bxi^\trans M \bxi$.
\end{theorem}

We do not pursue a high-degree variant of Theorem~\ref{thm:qf} for it seems that such a result can be attained through ideas unveiled in our proof of Theorem~\ref{thm:qf} and an adequate adaptation of a  decoupling argument seen in the work of Lovett~\cite{Lovett}.

\subsection{Proof of Theorem~\ref{thm:qf}}

\noindent
{\bf Assumptions.} Let $M \in\R^{n\times n}$ be a non-zero matrix. Any matrix $M$ can be uniquely decomposed into its symmetric and anti-symmetric parts 
$$
M=
\frac{1}{2}(M+M^\trans) + \frac{1}{2}(M-M^\trans).
$$ 
As $\frac{1}{2}\bxi^{\trans} (M-M^\trans) \bxi = 0$ holds for any $\bxi\in\{\pm 1\}^n$, we may take $M$ to be symmetric without loss of generality.
Next, writing $\diag(M)$ to denote the diagonal matrix whose main diagonal is that of $M$, it follows that 
\begin{align*}
 \bxi^\trans M \bxi = \bxi^\trans\diag(M)\bxi  + \bxi^\trans(M - \diag(M))\bxi
 =
\trace(M) + \bxi^\trans(M - \diag(M))\bxi,
\end{align*}
holds for any $\bxi \in \{\pm 1\}^n$. 
Consequently, 
\begin{align*}
\res_{x}^M(\bxi)=\res_{x-\trace(M)}^{M-\diag(M)}(\bxi),
\end{align*}
holds for any $x\in\R$ and any $\bxi \in \{\pm 1\}^n$. We may thus assume, without loss of generality, that the main diagonal of $M$ is zero. 

\bigskip
\noindent
{\bf Small-ball probabilities.}
As in the bilinear case, performing $\ress$ flips on $\bxi$ may alter the value of $\bxi ^\trans M \bxi$ by at most \begin{align*}
\max_{\bxi'\in N_{k}(\bxi)} |(\bxi') ^\trans M \bxi' - \bxi ^\trans M \bxi| 
& \leq 
2r\max_{\bxi'\in N_{k}(\bxi)}\norm{M(\bxi'-\bxi)}_\infty 
\\&\leq 
2r\big(\norm{M\bxi}_\infty + 2\min\{r,\norm{M}_{\infty,0}\}  \norm{M}_\infty \big).
\end{align*}
We may then write that
\begin{align}
\label{eq:qf-x-infty}
\sup_{x\in\R}\PROB\left\{\res_x^M(\bxi) \leq r\right\}
\leq
\sup_{x\in\R} \PROBp{|\bxi ^\trans M \bxi - x | \leq 2r\big(\norm{M\bxi}_\infty + 2\min\{r,\norm{M}_{\infty,0}\}  \norm{M}_\infty )}.
\end{align}

Recalling that each member $X_{\bt}$ of the random process $\bX$, defined in~\eqref{eq::gauss-proc}, is sub-gaussian with parameter $c\norm{\bt}_2^2$ (for some constant $c$) and that $\sup_{\bt \in \mathrm{Rows}(M)} X_{\bt} = \norm{M\bxi}_\infty $, allows us to bound the latter using the following tail inequality. 

\begin{theorem}\label{thm::Dudley-tail}{\em (Dudley's maximal tail inequality - abridged~\cite[Lemma~5.2]{van2014probability})}\\
Let $T \subseteq \R^n$ be finite and let $(Y_{\bt})_{\bt \in T}$ be a random process such that $Y_{\bt}$ is sub-gaussian with parameter $K^2$ for every $\bt \in T$. Then, 
$$
\PROBp{\sup_{\bt \in T} Y_{\bt} \geq K\sqrt{2 \log |T|} + x} \leq \exp\left(-\frac{x^2}{2K^2}\right)
$$
holds for every $x \geq 0$. 
\end{theorem}

\noindent
Using Theorem~\ref{thm::Dudley-tail}, we may now write that for a sufficiently large constant $C$, 
\begin{align*}
    \PROBp{\norm{M\bxi}_{\infty} > \sqrt{C^2\log n\norm{M}_{\infty,2}^2}}
    \leq  \frac{c_4}{n}
\end{align*}
holds. Setting
\begin{align*}
\eps_{M}(r) = 2r\sqrt{C^2\log n}\norm{M}_{\infty,2}+ 4r\min\{r,\norm{M}_{\infty,0}\}\norm{M}_\infty,
\end{align*}
we may rewrite~\eqref{eq:qf-x-infty} as to read
\begin{align}
\label{eq:qf-x}
\sup_{x\in\R}\PROB\left\{\res_x^M(\bxi) \leq r\right\}
\leq
\sup_{x \in \R} \PROB\left\{|\bxi ^\trans M \bxi - x | \leq \eps_{M}(r)\right\}
+\frac{c_4}{n}
\end{align}

\bigskip
\noindent
{\bf Decoupling.} To bound the first term appearing on the right hand side of \eqref{eq:qf-x} we reduce the small-ball probability for a quadartic form to that of a bilinear form through a decoupling argument. Let $I_1,I_2\subseteq [n]$ be an arbitrary partition of $[n]$. 
Write $\bxi_{(1)}$ and $\bxi_{(2)}$ to denote the restriction of $\bxi$ to the indices in $I_1$ and $I_2$ respectively.
Without loss of generality, we write $\bxi=(\bxi_{(1)},\bxi_{(2)})$ and set
\begin{align*}
\qf(\bxi_{(1)},\bxi_{(2)}) 
:= \bxi^\trans M \bxi    
= (\bxi_{(1)},\bxi_{(2)})^\trans M (\bxi_{(1)},\bxi_{(2)}).
\end{align*}
For a fixed $x \in \R$
 and $\eps>0$, 
define the event
\begin{align*}
\event(\bxi_{(1)},\bxi_{(2)}) := \event_{x,\eps}(\bxi_{(1)},\bxi_{(2)}):= \left\{|\qf(\bxi_{(1)},\bxi_{(2)}) -x|\leq \eps\right\}.
\end{align*}
Let $\bxi'_{(1)}$ and $\bxi'_{(2)}$ be independent copies of $\bxi_{(1)}$ and $\bxi_{(2)}$. 
Then,
\begin{align*}
\PROB\left\{\event(\bxi_{(1)},\bxi_{(2)})\right\}
&\leq 
\PROB\left\{\event(\bxi_{(1)},\bxi_{(2)}) \wedge \event(\bxi'_{(1)},\bxi_{(2)})  \wedge \event(\bxi_{(1)},\bxi'_{(2)}) \wedge \event(\bxi'_{(1)},\bxi'_{(2)}) \right\}^{1/4}
\end{align*}
holds by~\cite[Lemma 4.7]{CTV06}. 
Note that 
\begin{align*}
\event(\bxi_{(1)},\bxi_{(2)})  \wedge \event(\bxi'_{(1)},\bxi_{(2)}) &=
\left\{|\qf(\bxi_{(1)},\bxi_{(2)}) -x|\leq \eps\right\}
\wedge 
\left\{|\qf(\bxi'_{(1)},\bxi_{(2)}) -x|\leq \eps\right\}
\end{align*}
implies the event
\begin{align*}
\left\{|\qf(\bxi_{(1)},\bxi_{(2)}) -\qf(\bxi'_{(1)},\bxi_{(2)}) |\leq 2\eps\right\},
\end{align*}
and similarly to the second pair of events $\event(\bxi_{(1)},\bxi'_{(2)})  \wedge \event(\bxi'_{(1)},\bxi'_{(2)})$.
So the following inclusion of events holds
\begin{align*}
\event(\bxi_{(1)},\bxi_{(2)}) & \wedge \event(\bxi'_{(1)},\bxi_{(2)}) 
\wedge \event(\bxi_{(1)},\bxi'_{(2)}) \wedge \event(\bxi'_{(1)},\bxi'_{(2)}) \\ 
&\subseteq 
\left\{|\qf(\bxi_{(1)},\bxi_{(2)}) -\qf(\bxi'_{(1)},\bxi_{(2)}) |\leq 2\eps\right\}
\,\wedge\, 
\left\{|\qf(\bxi_{(1)},\bxi'_{(2)}) -\qf(\bxi'_{(1)},\bxi'_{(2)}) |\leq 2\eps\right\}\\
&\subseteq
\left\{|\qf(\bxi_{(1)},\bxi_{(2)}) -\qf(\bxi'_{(1)},\bxi_{(2)}) - \qf(\bxi_{(1)},\bxi'_{(2)}) +\qf(\bxi'_{(1)},\bxi'_{(2)}) |\leq 4\eps\right\}
\\ &= \left\{|R(I_1,I_2)|\leq 4\eps\right\},
\end{align*}
where we defined
\begin{align*}
R(I_1,I_2) & :=
\qf(\bxi_{(1)},\bxi_{(2)}) -\qf(\bxi'_{(1)},\bxi_{(2)}) - \qf(\bxi_{(1)},\bxi'_{(2)}) +\qf(\bxi'_{(1)},\bxi'_{(2)}). 
\end{align*}
Following~\cite{CTV06}, a straight forward calculation establishes that
\begin{align*}
R(I_1,I_2) & =
2\sum_{i\in I_1}\sum_{j\in I_2} M_{ij}(\xi_i -\xi'_i)(\xi_j -\xi'_j).
\end{align*}
The random variables $z_i=\big((\xi_i -\xi'_i)/2\big)_{i \in [n]}$ are i.i.d. lazy Rademacher random variables. 
Write $\bm z_{(1)} :=(z_i)_{i\in I_1}$ and $\bm z_{(2)}:=(z_i)_{i\in I_2}$ and let
$M_{(1,2)}$ denote the $|I_1|\times|I_2|$ submatrix of $M$ with rows in $I_1$ and columns in $I_2$.
Then, 
\begin{align*}
R(I_1,I_2)  = \frac{1}{2}\bm z_{(1)}^\trans M_{(1,2)} \bm z_{(2)}
\end{align*}
is a decoupled bilinear form. 
Hence,
the first term appearing on the right hand side of \eqref{eq:qf-x} is bounded by
\begin{align*}
\PROBp{|\bxi ^\trans M \bxi - x | \leq \eps_{M}(r)}
&\leq 
\PROBp{|R(I_1,I_2)|\leq 4\eps_{M}(r)}^{{1}/{4}}
\\&=\PROBp{|\bm z_{(1)}^\trans M_{(1,2)} \bm z_{(2)}| 
\leq 2\eps_{M}(r)}^{{1}/{4}}.
\end{align*}

In what follows, we insist on a partition $(I_1,I_2)$ of $[n]$  satisfying
\begin{equation}\label{eq::partition}
\norm{M_{(1,2)}}_\frob^2 \geq \frac{1}{8}\norm{M}_{\frob}^2;
\end{equation}
such a partition always exists by the pigeonhole principle.
To see this, consider a random partition $I_1(\delta) = \{i\in [n]: \delta_i=1\}$ and $I_2(\delta)=[n]\setminus I_1(\delta)$ where $(\delta_i)_{i\in[n]}$ are i.i.d.\ balanced Bernoulli random variables i.e. $\PROB\{\delta_i=1\}=1/2=\PROB\{\delta_i=0\}$. Then,  
$\EXP_\delta\{\delta_i(1-\delta_j)\}=\chr{i\neq j}/4$.
which in turn allows us to write  
\begin{align*}
\EXP_\delta\left\{\norm{M_{(1,2)}}_\frob^2\right\}
&=
\EXP_\delta\left\{\sum_{i\in I_1(\delta)}\sum_{j\in I_2(\delta)} M_{ij}^2 \right\}
=
\frac{1}{2}\EXP_\delta\left\{\sum_{i\in [n]}\sum_{j\in [n]} \delta_i(1-\delta_j)M_{ij}^2\right\}
=
\frac{1}{8}\norm{M}_{\frob}^2.
\end{align*}
In addition,  since $\EXP\{z_iz_j\} = 4\chr{i=j}$,
\begin{align*}
\EXP\left\{\norm{M_{(1,2)}\bm z_{(2)}}_2^2\right\} 
 &=
\sum_{i\in I_1}\sum_{j\in I_2} \sum_{\ell\in I_2} M_{ij} M_{i\ell}\EXP\left\{z_j z_\ell\right\} 
=
4\sum_{i\in I_1}\sum_{j\in I_2} M_{ij}^2
=
4\norm{M_{(1,2)}}_{\frob}^2
\geq \frac{1}{2}\norm{M}_{\frob}^2
\end{align*}
and the existence of a partition $(I_1,I_2)$ of $[n]$ satisfying~\eqref{eq::partition} is established.

Lemma \ref{lem:small_ball_bilinear} then asserts that
\begin{align*}
\PROB\left\{|\bm z_{(1)}^\trans M_{(1,2)} \bm z_{(2)}| \leq 2\eps_{M}(r)\right\} &\leq
\frac{2\eps_{M}(r)}{\norm{M_{(1,2)}}_\frob} + \exp{\left(-C \sr(M_{(1,2})\right)}
\\&\leq
\frac{4\eps_{M}(r)}{\norm{M}_\frob} + \exp{\left(-8C\sr(M)\right)},
\end{align*}
where in the last inequality we rely on $\norm{M_{(1,2)}}_\frob^2\geq\norm{M}_\frob^2 /8$ and the fact that
$\norm{M_{(1,2)}}_2 \leq \norm{M}_2$.
To see the latter property, note that
\begin{align*}
\norm{M_{(1,2)}}_2^2   & =
\sup_{\bm v: \|\bm v\|_2=1} \norm{M_{(1,2)}\bm v}_2^2
\\&\leq
\sup_{\bm v: \|\bm v\|_2=1}
 \left\|\begin{pmatrix}
      0 & M_{(1,2)}
      \\ 0  & M_{(2,2)}
 \end{pmatrix}
\begin{pmatrix}
\bm 0 \\ \bm  v
\end{pmatrix}
\right\|_2^2
 \leq
\sup_{\bm v: \|\bm v\|_2=1}
 \left\|M \bm v
\right\|_2^2
= 
\left\|M\right\|_2^2,
\end{align*}
where the first inequality holds since by adding the value $\norm{M_{(2,2)}\bm v}_2^2$ we only increase the 2-norm. The second inequality holds since we allow $\bm v$ to span a larger space.

Putting everything together yields, 
\begin{align*}
\sup_{x\in\R}& \PROB\left\{\res_x^M(\bxi) \leq r\right\} \\
& \leq \left(\frac{4\eps_{M}(r)}{\norm{M}_\frob} + \exp{\left(-8C\sr(M)\right)}\right)^{1/4}
+ \frac{c_4}{n}
\\
&\leq \left(\frac{c_1 k\sqrt{\log n}\norm{M}_{\infty,2}}{\norm{M}_\frob}
+
\frac{c_
2k\min\{k,\norm{M}_{\infty,0}\}\norm{M}_\infty}{\norm{M}_\frob}
+
\exp{\left(-c_3\sr(M)\right)}
\right)^{1/4}
+ \frac{c_4}{n}.
\end{align*}

\end{document}